\documentclass[12pt]{amsart}
\usepackage{amsmath,amssymb,amsbsy,amsfonts,amsthm,latexsym,
                        amsopn,amstext,amsxtra,euscript,amscd,mathrsfs,color,bm}
                        
\usepackage{float} 
\usepackage[english]{babel}
\usepackage{url}
\usepackage[colorlinks,linkcolor=blue,anchorcolor=blue,citecolor=blue,backref=page]{hyperref}

\begin{document}

\renewcommand*{\backref}[1]{}
\renewcommand*{\backrefalt}[4]{%
    \ifcase #1 (Not cited.)%
    \or        (p.\,#2)%
    \else      (pp.\,#2)%
    \fi}     
    
%%%%%%%%%%%%%%%%%%%%%%%%%%%%%%%%%%%%%%%%%%%%%%%%%%%%%%%%%%%%%%%%%%%%%%
\newtheorem{theorem}{Theorem}
\newtheorem{lemma}[theorem]{Lemma}
\newtheorem{example}[theorem]{Example}
\newtheorem{algol}{Algorithm}
\newtheorem{corollary}[theorem]{Corollary}
\newtheorem{prop}[theorem]{Proposition}
\newtheorem{proposition}[theorem]{Proposition}
\newtheorem{problem}[theorem]{Problem}
\newtheorem{conjecture}[theorem]{Conjecture}

\theoremstyle{remark}
\newtheorem{definition}[theorem]{Definition}
\newtheorem{question}[theorem]{Question}
\newtheorem{remark}[theorem]{Remark}
\newtheorem*{acknowledgement}{Acknowledgements}

\numberwithin{equation}{section}
\numberwithin{theorem}{section}
\numberwithin{table}{section}
\numberwithin{figure}{section}

\allowdisplaybreaks

%%%%%%%%%%%%%%%%%%%%%%%%%%%%%%%%%%%%%%%%%%%%%%%%%%%%%%%%%%%%%%%%%%%%%%
\definecolor{olive}{rgb}{0.3, 0.4, .1}
\definecolor{dgreen}{rgb}{0.,0.5,0.}

\def\cc#1{\textcolor{red}{#1}} 

\newcommand{\commAB}[1]{\marginpar{%
\begin{color}{dgreen}
\vskip-\baselineskip %raise the marginpar a bit
\raggedright\footnotesize
\itshape\hrule \smallskip AB: #1\par\smallskip\hrule\end{color}}\ignorespaces}
 
\newcommand{\commAO}[1]{\marginpar{%
\begin{color}{blue}
\vskip-\baselineskip %raise the marginpar a bit
\raggedright\footnotesize
\itshape\hrule \smallskip AO: #1\par\smallskip\hrule\end{color}}\ignorespaces}

\newcommand{\commIS}[1]{\marginpar{%
\begin{color}{magenta}
\vskip-\baselineskip %raise the marginpar a bit
\raggedright\footnotesize
\itshape\hrule \smallskip IS: #1\par\smallskip\hrule\end{color}}\ignorespaces}

\newcommand{\commJS}[1]{\marginpar{%
\begin{color}{red}
\vskip-\baselineskip %raise the marginpar a bit
\raggedright\footnotesize
\itshape\hrule \smallskip JS: #1\par\smallskip\hrule\end{color}}\ignorespaces}
\def\cA{{\mathcal A}}
\def\cB{{\mathcal B}}
\def\cC{{\mathcal C}}
\def\cD{{\mathcal D}}
\def\cE{{\mathcal E}}
\def\cF{{\mathcal F}}
\def\cG{{\mathcal G}}
\def\cH{{\mathcal H}}
\def\cI{{\mathcal I}}
\def\cJ{{\mathcal J}}
\def\cK{{\mathcal K}}
\def\cL{{\mathcal L}}
\def\cM{{\mathcal M}}
\def\cN{{\mathcal N}}
\def\cO{{\mathcal O}}
\def\cP{{\mathcal P}}
\def\cQ{{\mathcal Q}}
\def\cR{{\mathcal R}}
\def\cS{{\mathcal S}}
\def\cT{{\mathcal T}}
\def\cU{{\mathcal U}}
\def\cV{{\mathcal V}}
\def\cW{{\mathcal W}}
\def\cX{{\mathcal X}}
\def\cY{{\mathcal Y}}
\def\cZ{{\mathcal Z}}

\def\C{\mathbb{C}}
\def\K{\mathbb{K}}
\def\Z{\mathbb{Z}}
\def\R{\mathbb{R}}
\def\Q{\mathbb{Q}}
\def\N{\mathbb{N}}
\def\L{\mathbb{L}}
\def\M{\textsf{M}}
\def\U{\mathbb{U}}
\def\P{\mathbb{P}}
\def\A{\mathbb{A}}
\def\fp{\mathfrak{p}}
\def\fq{\mathfrak{q}}
\def\n{\mathfrak{n}}
\def\X{\mathcal{X}}
\def\x{\textrm{\bf x}}
\def\w{\textrm{\bf w}}
\def\ovQ{\overline{\Q}}
\def \Kab{\K^{\mathrm{ab}}}
\def \Qab{\Q^{\mathrm{ab}}}
\def \Qtr{\Q^{\mathrm{tr}}}
\def \Kc{\K^{\mathrm{c}}}
\def \Qc{\Q^{\mathrm{c}}}
\def\ZK{\Z_\K}
\def\ZKS{\Z_{\K,\cS}}
\def\ZKSf{\Z_{\K,\cS_f}}
\def\RSfG{R_{\cS_{f,\Gamma}}}

\def\S{\mathcal{S}}
\def\vec#1{\mathbf{#1}}
\def\ov#1{{\overline{#1}}}
\def\Sp{{\operatorname{S}}}
\def\Gm{\G_{\textup{m}}}
\def\fA{{\mathfrak A}}
\def\fB{{\mathfrak B}}

\def\house#1{{%
    \setbox0=\hbox{$#1$}
    \vrule height \dimexpr\ht0+1.4pt width .5pt depth \dp0\relax
    \vrule height \dimexpr\ht0+1.4pt width \dimexpr\wd0+2pt depth \dimexpr-\ht0-1pt\relax
    \llap{$#1$\kern1pt}
    \vrule height \dimexpr\ht0+1.4pt width .5pt depth \dp0\relax}}

%%%%%%%%%%%%%%%%%%%%%%%%%%%%%%%%%%%%%%%%%%%%%%%%%%%%%%%%%%%%%%%%%%%%%%

%%%%%%%% Set Up Environment for Notation %%%%%%%%%%%%%%
% This is currently set to allow quite wide items to be defined
\newenvironment{notation}[0]{%
  \begin{list}%
    {}%
    {\setlength{\itemindent}{0pt}
     \setlength{\labelwidth}{1\parindent}
     \setlength{\labelsep}{\parindent}
     \setlength{\leftmargin}{2\parindent}
     \setlength{\itemsep}{0pt}
     }%
   }%
  {\end{list}}

%%%%%%%% Set Up Environment for Parts in Theorems %%%%%%%%%%%%%%
\newenvironment{parts}[0]{%
  \begin{list}{}%
    {\setlength{\itemindent}{0pt}
     \setlength{\labelwidth}{1.5\parindent}
     \setlength{\labelsep}{.5\parindent}
     \setlength{\leftmargin}{2\parindent}
     \setlength{\itemsep}{0pt}
     }%
   }%
  {\end{list}}
% Use \Part{(a)}, instead of \item[(a)], to ensure upright font
\newcommand{\Part}[1]{\item[\upshape#1]}

%%%%%%%% Set Up Macro for Cases %%%%%%%%%%%%%%
\def\Case#1#2{%
\smallskip\paragraph{\textbf{\boldmath Case #1: #2.}}\hfil\break\ignorespaces}

%%%%%%%%%%%%%%%%%%
% Greek Alphabet %
%%%%%%%%%%%%%%%%%%
\renewcommand{\a}{\alpha}
\renewcommand{\b}{\beta}
\newcommand{\g}{\gamma}
\renewcommand{\d}{\delta}
\newcommand{\e}{\epsilon}
\newcommand{\f}{\varphi}
\newcommand{\fhat}{\hat\varphi}
\newcommand{\bfphi}{{\boldsymbol{\f}}}
\renewcommand{\l}{\lambda}
\renewcommand{\k}{\kappa}
\newcommand{\lhat}{\hat\lambda}
\newcommand{\bfmu}{{\boldsymbol{\mu}}}
\renewcommand{\o}{\omega}
\renewcommand{\r}{\rho}
\newcommand{\rbar}{{\bar\rho}}
\newcommand{\s}{\sigma}
\newcommand{\sbar}{{\bar\sigma}}
\renewcommand{\t}{\tau}
\newcommand{\z}{\zeta}

%\newcommand{\D}{\Delta}
%\newcommand{\G}{\Gamma}
%\newcommand{\F}{\Phi}
%\renewcommand{\L}{\Lambda}

%%%%%%%%%%%%%%%%%%%%
% Fraktur Alphabet %
%%%%%%%%%%%%%%%%%%%%
\newcommand{\ga}{{\mathfrak{a}}}
\newcommand{\gb}{{\mathfrak{b}}}
\newcommand{\gn}{{\mathfrak{n}}}
\newcommand{\gp}{{\mathfrak{p}}}
\newcommand{\gP}{{\mathfrak{P}}}
\newcommand{\gq}{{\mathfrak{q}}}

%%%%%%%%%%%%%%%%%%%
% Barred Alphabet %
%%%%%%%%%%%%%%%%%%%
\newcommand{\Abar}{{\bar A}}
\newcommand{\Ebar}{{\bar E}}
\newcommand{\kbar}{{\bar k}}
\newcommand{\Kbar}{{\bar K}}
\newcommand{\Pbar}{{\bar P}}
\newcommand{\Sbar}{{\bar S}}
\newcommand{\Tbar}{{\bar T}}
\newcommand{\gbar}{{\bar\gamma}}
\newcommand{\lbar}{{\bar\lambda}}
\newcommand{\ybar}{{\bar y}}
\newcommand{\phibar}{{\bar\f}}

%%%%%%%%%%%%%%%%%%%%%%%%%
% Calligraphic Alphabet %
%%%%%%%%%%%%%%%%%%%%%%%%%
\newcommand{\Acal}{{\mathcal A}}
\newcommand{\Bcal}{{\mathcal B}}
\newcommand{\Ccal}{{\mathcal C}}
\newcommand{\Dcal}{{\mathcal D}}
\newcommand{\Ecal}{{\mathcal E}}
\newcommand{\Fcal}{{\mathcal F}}
\newcommand{\Gcal}{{\mathcal G}}
\newcommand{\Hcal}{{\mathcal H}}
\newcommand{\Ical}{{\mathcal I}}
\newcommand{\Jcal}{{\mathcal J}}
\newcommand{\Kcal}{{\mathcal K}}
\newcommand{\Lcal}{{\mathcal L}}
\newcommand{\Mcal}{{\mathcal M}}
\newcommand{\Ncal}{{\mathcal N}}
\newcommand{\Ocal}{{\mathcal O}}
\newcommand{\Pcal}{{\mathcal P}}
\newcommand{\Qcal}{{\mathcal Q}}
\newcommand{\Rcal}{{\mathcal R}}
\newcommand{\Scal}{{\mathcal S}}
\newcommand{\Tcal}{{\mathcal T}}
\newcommand{\Ucal}{{\mathcal U}}
\newcommand{\Vcal}{{\mathcal V}}
\newcommand{\Wcal}{{\mathcal W}}
\newcommand{\Xcal}{{\mathcal X}}
\newcommand{\Ycal}{{\mathcal Y}}
\newcommand{\Zcal}{{\mathcal Z}}

%%%%%%%%%%%%%%%%%%%%%%%%%%%%
% Blackboard Bold Alphabet %
%%%%%%%%%%%%%%%%%%%%%%%%%%%%
\renewcommand{\AA}{\mathbb{A}}
\newcommand{\BB}{\mathbb{B}}
\newcommand{\CC}{\mathbb{C}}
\newcommand{\FF}{\mathbb{F}}
\newcommand{\GG}{\mathbb{G}}
\newcommand{\KK}{\mathbb{K}}
\newcommand{\NN}{\mathbb{N}}
\newcommand{\PP}{\mathbb{P}}
\newcommand{\QQ}{\mathbb{Q}}
\newcommand{\RR}{\mathbb{R}}
\newcommand{\ZZ}{\mathbb{Z}}

%%%%%%%%%%%%%%%%%%%%%%%%%%
% Boldface Math Alphabet %
%%%%%%%%%%%%%%%%%%%%%%%%%%
\newcommand{\bfa}{{\boldsymbol a}}
\newcommand{\bfb}{{\boldsymbol b}}
\newcommand{\bfc}{{\boldsymbol c}}
\newcommand{\bfd}{{\boldsymbol d}}
\newcommand{\bfe}{{\boldsymbol e}}
\newcommand{\bff}{{\boldsymbol f}}
\newcommand{\bfg}{{\boldsymbol g}}
\newcommand{\bfi}{{\boldsymbol i}}
\newcommand{\bfj}{{\boldsymbol j}}
\newcommand{\bfk}{{\boldsymbol k}}
\newcommand{\bfm}{{\boldsymbol m}}
\newcommand{\bfp}{{\boldsymbol p}}
\newcommand{\bfr}{{\boldsymbol r}}
\newcommand{\bfs}{{\boldsymbol s}}
\newcommand{\bft}{{\boldsymbol t}}
\newcommand{\bfu}{{\boldsymbol u}}
\newcommand{\bfv}{{\boldsymbol v}}
\newcommand{\bfw}{{\boldsymbol w}}
\newcommand{\bfx}{{\boldsymbol x}}
\newcommand{\bfy}{{\boldsymbol y}}
\newcommand{\bfz}{{\boldsymbol z}}
\newcommand{\bfA}{{\boldsymbol A}}
\newcommand{\bfF}{{\boldsymbol F}}
\newcommand{\bfB}{{\boldsymbol B}}
\newcommand{\bfD}{{\boldsymbol D}}
\newcommand{\bfG}{{\boldsymbol G}}
\newcommand{\bfI}{{\boldsymbol I}}
\newcommand{\bfM}{{\boldsymbol M}}
\newcommand{\bfP}{{\boldsymbol P}}
\newcommand{\bfX}{{\boldsymbol X}}
\newcommand{\bfY}{{\boldsymbol Y}}
\newcommand{\bfzero}{{\boldsymbol{0}}}
\newcommand{\bfone}{{\boldsymbol{1}}}

%%%%%%%%%%%%%%%%%%%%%%%%%%%%%%
% Miscellaneous New Commands %
%%%%%%%%%%%%%%%%%%%%%%%%%%%%%%
\newcommand{\aff}{{\textup{aff}}}
\newcommand{\Aut}{\operatorname{Aut}}
\newcommand{\Berk}{{\textup{Berk}}}
\newcommand{\Birat}{\operatorname{Birat}}
\newcommand{\characteristic}{\operatorname{char}}
\newcommand{\codim}{\operatorname{codim}}
\newcommand{\Crit}{\operatorname{Crit}}
\newcommand{\critwt}{\operatorname{critwt}} % valency of a portrait
\newcommand{\Cycle}{\operatorname{Cycles}}
\newcommand{\diag}{\operatorname{diag}}
\newcommand{\Disc}{\operatorname{Disc}}
\newcommand{\Div}{\operatorname{Div}}
\newcommand{\Dom}{\operatorname{Dom}}
\newcommand{\End}{\operatorname{End}}
\newcommand{\ExtOrbit}{\mathcal{EO}} %% Extended orbit
\newcommand{\Fbar}{{\bar{F}}}
\newcommand{\Fix}{\operatorname{Fix}}
\newcommand{\FOD}{\operatorname{FOD}}
\newcommand{\FOM}{\operatorname{FOM}}
\newcommand{\Gal}{\operatorname{Gal}}
\newcommand{\genus}{\operatorname{genus}}
\newcommand{\GITQuot}{/\!/}
\newcommand{\GL}{\operatorname{GL}}
\newcommand{\GR}{\operatorname{\mathcal{G\!R}}}
\newcommand{\Hom}{\operatorname{Hom}}
\newcommand{\Index}{\operatorname{Index}}
\newcommand{\Image}{\operatorname{Image}}
\newcommand{\Isom}{\operatorname{Isom}}
\newcommand{\hhat}{{\hat h}}
\newcommand{\Ker}{{\operatorname{ker}}}
\newcommand{\Ksep}{K^{\textup{sep}}}  %% separable closure of K
\newcommand{\lcm}{{\operatorname{lcm}}}
\newcommand{\LCM}{{\operatorname{LCM}}}
\newcommand{\Lift}{\operatorname{Lift}}
\newcommand{\limstar}{\lim\nolimits^*}
\newcommand{\limstarn}{\lim_{\hidewidth n\to\infty\hidewidth}{\!}^*{\,}}
\newcommand{\llog}{\log\log}
\newcommand{\logplus}{\log^{\scriptscriptstyle+}}
\newcommand{\Mat}{\operatorname{Mat}}
\newcommand{\maxplus}{\operatornamewithlimits{\textup{max}^{\scriptscriptstyle+}}}
\newcommand{\MOD}[1]{~(\textup{mod}~#1)}
\newcommand{\Mor}{\operatorname{Mor}}
\newcommand{\Moduli}{\mathcal{M}}
\newcommand{\Norm}{{\operatorname{\mathsf{N}}}}
\newcommand{\notdivide}{\nmid}
\newcommand{\normalsubgroup}{\triangleleft}
\newcommand{\NS}{\operatorname{NS}}
\newcommand{\onto}{\twoheadrightarrow}
\newcommand{\ord}{\operatorname{ord}}
\newcommand{\Orbit}{\mathcal{O}}
\newcommand{\Per}{\operatorname{Per}}
\newcommand{\Perp}{\operatorname{Perp}}
\newcommand{\PrePer}{\operatorname{PrePer}}
\newcommand{\PGL}{\operatorname{PGL}}
\newcommand{\Pic}{\operatorname{Pic}}
\newcommand{\Prob}{\operatorname{Prob}}
\newcommand{\Proj}{\operatorname{Proj}}
\newcommand{\Qbar}{{\bar{\QQ}}}
\newcommand{\rank}{\operatorname{rank}}
\newcommand{\Rat}{\operatorname{Rat}}
\newcommand{\Res}{{\operatorname{Res}}}
\newcommand{\Resultant}{\operatorname{Res}}
\renewcommand{\setminus}{\smallsetminus}
\newcommand{\sgn}{\operatorname{sgn}}
\newcommand{\SL}{\operatorname{SL}}
\newcommand{\Span}{\operatorname{Span}}
\newcommand{\Spec}{\operatorname{Spec}}
\renewcommand{\ss}{{\textup{ss}}}
\newcommand{\stab}{{\textup{stab}}}
\newcommand{\Stab}{\operatorname{Stab}}
\newcommand{\Support}{\operatorname{Supp}}
\newcommand{\Sym}{\operatorname{Sym}}  %% Symmetric group
\newcommand{\tors}{{\textup{tors}}}
\newcommand{\Trace}{\operatorname{Trace}}
\newcommand{\trianglebin}{\mathbin{\triangle}} % symmetric set difference
\newcommand{\tr}{{\textup{tr}}} % for K/k trace
\newcommand{\UHP}{{\mathfrak{h}}}    % Upper half plane
\newcommand{\Wander}{\operatorname{Wander}}
\newcommand{\<}{\langle}
\renewcommand{\>}{\rangle}

\newcommand{\pmodintext}[1]{~\textup{(mod}~#1\textup{)}}
\newcommand{\ds}{\displaystyle}
\newcommand{\longhookrightarrow}{\lhook\joinrel\longrightarrow}
\newcommand{\longonto}{\relbar\joinrel\twoheadrightarrow}
\newcommand{\SmallMatrix}[1]{%
  \left(\begin{smallmatrix} #1 \end{smallmatrix}\right)}
  %%%%%%%%%%%%%%%%%%%%%%%%%%%%%%%%%%

\title[Multiplicative dependence of rational functions]
{Multiplicative dependence among iterated values of rational functions modulo finitely generated groups}

\author[A. B\' erczes] {Attila B\' erczes}
\address{Institute of Mathematics, University of Debrecen, H-4010 Debrecen, P.O. BOX 12, Hungary}
\email{berczesa@science.unideb.hu}

\author[A. Ostafe] {Alina Ostafe}
\address{School of Mathematics and Statistics, University of New South Wales, Sydney NSW 2052, Australia}
\email{alina.ostafe@unsw.edu.au}

\author[I. E. Shparlinski] {Igor E. Shparlinski}
\address{School of Mathematics and Statistics, University of New South Wales, Sydney NSW 2052, Australia}
\email{igor.shparlinski@unsw.edu.au}

\author[J. H. Silverman]{Joseph H. Silverman}
\address{Mathematics Department, Box 1917, 
  Brown University, Providence, RI 02912 USA}
\email{jhs@math.brown.edu}

\subjclass[2010]{37P15, 11R27}

\keywords{arithmetic dynamics, multiplicative dependence} 

\begin{abstract} 
We study multiplicative dependence between elements in orbits of
algebraic dynamical systems over number fields modulo a finitely
generated multiplicative subgroup of the field.  We obtain a series of
results, many of which may be viewed as a blend of Northcott's
theorem on boundedness of preperiodic points and Siegel's theorem on
finiteness of solutions to $S$-unit equations.
\end{abstract}

\maketitle

%%%%%%%%%%%%%%%%%%%%%%%%%%%%%%%%%%%%%%%%%%%%%%%%%%%%%%%%%%%%%%%%%%%%%%
\section{Introduction and statements of main results}

\subsection{Motivation} 
Let $\K/\QQ$ be a number field with algebraic closure~$\ov\K$,
and let~$f(X)\in\K(X)$ be a rational function of degree at last~$2$.
A famous theorem of Northcott~\cite{Northcott}
says that the set of~$\ov\K$-preperiodic points of~$f$ is a set of
bounded Weil height. In particular, the set of $\K$-rational
preperiodic points is finite.  In~\cite{OSSZ1} this result is
extended to cover the case that points in an orbit are
multiplicatively dependent, rather than forcing them to be equal. In
this paper we extend this further to the case of points in an orbit
that are multiplicatively dependent modulo a finitely generated
subgroup~$\Gamma$ of~$\K^*$. Anticipating notation that is described
in Section~\ref{subsection:notation}, we are interested in
solutions~$(n,k,\alpha,r,s)$ to the relation
\[
  f^{(n+k)}(\a)^r\cdot f^{(k)}(\a)^s \in \Gamma,
\]
where~$n\ge1$,~$k\ge0$,~$\alpha\in\K$, and~$(r,s)\ne(0,0)$.  We view this
as a combination of Northcott's already cited theorem and Siegel's
theorem concerning integral points on affine curves, which at its
heart deals with integral solutions to equations of the form
$F(x,y)\in\Gamma$ for homogenous $F\in\K[X,Y]$. And indeed, a key tool in
the proof of several of our main results is a theorem on dynamical
Diophantine approximation (Lemma~\ref{lem:lvfnageehffna}) that
ultimately relies on Roth's theorem and is very much analogous to the
use of Diophantine approximation in the proof of Siegel's theorem.

\subsection{Notation and conventions}
\label{subsection:notation}
We now set the following notation, which remains fixed for the remainder of
this paper:
\begin{notation}
 \item[\textbullet]  $\K$ is a number field.
  \item[\textbullet]  $R_\K$ is the ring of algebraic integers of $\K$.
 \item[\textbullet]  $\ov\K$ is an algebraic closure of $\K$.
 \item[\textbullet]  $f(X)\in\K(X)$ a rational function of degree $d\ge2$.
 \item[\textbullet]  For $n\ge0$, we write $f^{(n)}(X)$ for the $n$th iterate of~$f$, i.e., 
  \[
    f^{(n)}(X) := \underbrace{f\circ f\circ\cdots\circ f}_{\text{$n$ copies}}(X).
  \]
 \item[\textbullet]  For $\a\in\PP^1(\KK)$, we write $\Orbit_f(\a)$
  for the (forward) orbit of~$\a$, i.e., 
 \[
   \Orbit_f(\a) := \left \{ f^{(n)}(\a):n\ge 0 \right\}.
 \]
 \item[\textbullet]  $\Gamma$ is  a finitely generated subgroup of $\K^*$.
 \item[\textbullet] $\PrePer(f)$ is  the set of preperiodic points of $f$ in $\P^1(\ov\K)$, i.e., 
   the set of points $\a\in\PP^1(\ov\K)$ such that $\Orbit_f(\a)$ is finite.
 \item[\textbullet] $\Wander_\K(f)$ is the complement of the set $\PrePer(f)$ in $\P^1(\K)$,  i.e.,
         the set $\P^1(\K)\setminus \bigl(\PrePer(f)\cap\K\bigr)$ of $\K$-rational wandering points for~$f$.
 \item[\textbullet] $\Z_{\ge r}$ denotes the set of integers $n \ge r$, where $r$ is a real number.  
\end{notation}

It is also convenient to define the function
\[
  \logplus t =\log\max\{t,1\}. 
\]

We use $M_\K$ to denote a complete set of inequivalent absolute values
on~$\K$, normalized so that the absolute Weil height
$h:\ov{\K}\to[0,\infty)$ is defined by
\begin{equation}
\label{eqn:defht}
  h(\b) = \sum_{v\in M_\K} \logplus \bigl( \|\beta\|_v \bigr),
\end{equation}
and we write~$M_\K^\infty$ and~$M_\K^0$ for, respectively, the set of
archimedean and non-archimedean absolute values
in~$M_\K$. See~\cite{HinSil,Lang} for further details on absolute
values and height functions.
  
 As in~\cite[Section~1.2]{Silv07}, for a rational function
 $f(X)\in\C(X)$ and $\alpha\in\C$ with $\alpha\neq \infty$ and
 $f(\alpha)\ne\infty$, we define the \textit{ramification index} of $f$
 at $\alpha$ as the order of $\alpha$ as a zero of the rational
 function $f(X)-f(\alpha)$, i.e.,
\[
   e_f(\alpha)=\ord_\alpha(f(X)-f(\alpha)).
\]
In particular, we say that $f$ is \emph{ramified at $\alpha$} if $
e_f(\alpha)\ge 2$, and \textit{totally ramified at $\alpha$} if
$e_f(\alpha)=\deg f$.  If $\alpha=\infty$ or $f(\alpha)=\infty$, we
define $e_f(\alpha)$ by choosing a linear fractional transformation
$L\in\PGL_2(\C)$ so that $\beta=L^{-1}(\alpha)$ satisfies $\beta\ne
\infty$ and $f_L(\beta)\ne \infty$, where $f_L=L^{-1}\circ f\circ L$, and then we set
 \[
   e_f(\alpha)= e_{f_L}(\beta).
 \]
It is an exercise using the chain rule to show that~$e_f(\alpha)$
does not depend on the choice of $L$;
cf.~\cite[Exercise~1.5]{Silv07}.
It is also an exercise to show that  the ramification index is
multiplicative under the composition, i.e., for rational functions
$f,g\in\C(X)$ and $\alpha\in\C \cup\{\infty\}$, we have
\begin{equation}
  \label{eqn:mult e}
  e_{g\circ f}(\alpha) = e_f(\alpha)e_g\bigl(f(\alpha)\bigr).
\end{equation}

\subsection{Main results}
In this section we describe the main results proven in this paper.  In
order to state our results, we define three sets of ``exceptional
values''. Our theorems characterize when these sets may be infinite.

\begin{definition}
With~$\K$,~$f$, and~$\Gamma$ as defined in
Section~\ref{subsection:notation}, and for integers~$r,s\in\ZZ$ and
real number~$\rho>0$, we define various sets of exceptional values:
\begin{align*}
 & \Ecal_\rho(\K,f,\Gamma,r,s)  =
    \left\{\begin{aligned}
      (n,k,\a,u) \in\Z_{\ge\rho} \times \Z_{\ge0}& \times \Wander_\K(f) \times \Gamma:\\
      &f^{(n+k)}(\a)^r  = uf^{(k)}(\a)^s \\
    \end{aligned} 
    \right\};\\
  & \Fcal_\rho(\K,f,\Gamma) =
  \bigl\{ (n,\alpha)\in\ZZ_{\ge\rho}\times\Wander_\K(f) : f^{(n)}(\a) \in \Gamma \bigr\}; \\
&  \Gcal(\K,f,\Gamma) =   \bigl\{ \alpha\in\K : f(\a) \in \Gamma \bigr\}.
\end{align*}
\end{definition}

Clearly the finiteness of $\Gcal(\K,f,\Gamma)$ is equivalent to the
finiteness of $f(K)\cap\Gamma$.  Thus, the first part of the following
result is given in~\cite[Proposition~1.5(a)]{KLSTYZ} in the form
\[
  \#\bigl(f(K)\cap\Gamma\bigr)=\infty
  \quad\Longrightarrow\quad
  \#f^{-1}\bigl(\{0,\infty\}\bigr)\le2, 
\]
see also~\cite[Corollary~2.2]{OstShp}. 
We observe that the functions in Theorem~\ref{theorem:fnainG}(a) below are
exactly the functions having this last property.

\begin{theorem}
\label{theorem:fnainG} We have:
\begin{parts}
\Part{(a)}
If the set $\Gcal(\K,f,\Gamma)$
is infinite, then~$f(X)$ has one of the following forms\textup:
\begin{align*}
f(X) &= a(X-b)^{\pm d} &&\text{with $a\ne0$,} \\
f(X) &= a(X-b)^d/(X-c)^d &&\text{with $a(b-c)\ne0$.}
\end{align*}
\Part{(b)}
If the set $\Fcal_2(\K,f,\Gamma)$
is infinite, then~$f(X)$ has  the form $f(X)=aX^{\pm d}$.
\end{parts}
\end{theorem}

So next we only deal with the case $rs\ne  0$. 

\begin{theorem}
\label{thm:fnkaufnars}
Let $r,s\in\Z$ with $rs\ne0$, and set
\[
  \rho = \frac{\log (|s|/|r|)}{\log d}+1.
\]
Assume that~$0$ is not a periodic point of~$f$. Then
\[
  \#\Ecal_\rho(\K,f,\Gamma,r,s)=\infty
  \quad\Longrightarrow\quad
  \#f^{-1}\bigl(\{0,\infty\}\bigr)\le2.
\]
\end{theorem}

For notational convenience, let
\begin{equation}
\label{eqn:EKfG}
\Ecal_\rho(\KK,f,\Gamma):=\Ecal_\rho(\KK,f,\Gamma,1,1)
\quad \text{and} \quad
\Ecal(\KK,f,\Gamma):=\Ecal_1(\KK,f,\Gamma). 
\end{equation} 

Setting $r=s=1$ in Theorem~\ref{thm:fnkaufnars}, so~$\rho=1$, gives a finiteness
result for the set $\Ecal(\KK,f,\Gamma)$,
which is the set of solutions to the equation
\begin{equation}
\begin{split}
  \label{eqn:fnkaufka}
  f^{(n+k)}(\a) & = uf^{(k)}(\a), \\
  &\quad  (n,k,\a,u) \in \Z_{\ge1} \times \Z_{\ge0} \times \Wander_\K(f) \times \Gamma.
\end{split}
\end{equation}
In this situation we are able to give a full classification
of the exceptional cases, i.e., a complete description of the
maps~$f$ for which~\eqref{eqn:fnkaufka} may have infinitely
many solutions.

\begin{theorem}
\label{thm:fnkaufna}
For the sets~\eqref{eqn:EKfG}, we have\textup:
\begin{parts}
\Part{(a)}
If $\Ecal(\KK,f,\Gamma)$ is infinite, then either $f(X)$ or $f(X^{-1})^{-1}$
has one of the following forms with $abc(b-c)\ne0$\textup:
\begin{align*}
&aX^{\pm d}, \quad
aX^d/(X-b)^{d-1}, \quad
aX(X-b)^{d-1}, \\
&aX/(X-b)^d, \quad
aX(X-b)^{d-1}/(X-c)^{d-1}.
\end{align*}
\Part{(b)}
If   $\Ecal_2(\KK,f,\Gamma)$ is
infinite, then~$f(X)$ has the form $f(X)=aX^{\pm d}$.
\end{parts}
\end{theorem}

\begin{remark}
An analysis similar to the proof of Theorem~$\ref{thm:fnkaufna}$ can be
used to describe all~$f(X)$ for which $\Ecal(\KK,f,\Gamma)$ may have
infinitely many four-tuples $(1,k,\a,u)$ with $k\ge2$.
\end{remark}

\begin{remark}
If $f(X)$ is a polynomial, it suffices to assume in
Theorem~$\ref{thm:fnkaufna}$ that~$0$ is not a periodic point of~$f$ to
ensure that~$\Ecal(\KK,f,\Gamma)$ is finite.  However, for rational functions
there are examples such as \text{$f(X)=(1-X)^2/X$} with~$0$ strictly
preperiodic and 
\[
  (1,0,1/(u+1),u^2)\in\Ecal(\KK,f,\Gamma)
  \quad\text{for all $u\in\Gamma$.}
\]
\end{remark}

In the case that $f\in\K[X]$ is a polynomial, we have the
following broad extension of Theorem~\ref{thm:fnkaufnars} in which the
exponents $r$ and $s$ are not necessarily fixed. This allows us to
bound all multiplicative dependences between $f^{(m)}(\alpha)$ and
$f^{(n)}(\alpha)$ modulo $\Gamma$.

\begin{theorem}
\label{thm:2multdep-arb}
Let $f\in\K[X]$ be a polynomial without multiple roots, of degree $d\ge 3$ or, if $d = 2$, we also assume that $f^{(2)}$ has no multiple roots. 
Assume that~$0$ is not a periodic point of~$f$. Let
$\Gamma\subseteq\K^*$ be a finitely generated subgroup.  Then there
are only finitely many elements $\alpha \in \K$ such that for distinct
integers $m, n\ge 1$, the values $f^{(m)}(\alpha)$ and
$f^{(n)}(\alpha)$ are multiplicatively dependent modulo $\Gamma$.
\end{theorem}

\begin{remark}
Note that Theorem~$\ref{thm:2multdep-arb}$ fails if we allow~$m$
or~$n$ to be~$0$, since for any $u\in\Gamma$ and any $m\ge1$, there is
a multiplicative relation $(f^{(m)}(u))^0\cdot f^{(0)}(u)=u\in\Gamma$.
\end{remark}

Finally, we present an independence result of a slightly different
type.

\begin{definition}
Let $k\ge1$.  A polynomial $F(T_1,\ldots,T_k)\in\K[T_1,\ldots,T_k]$ is
said to be a \emph{multilinear polynomial with split variables} if
there are scalars $c_1,\ldots,c_k\in\K^*$ and a disjoint partition
\[
J_1\cup J_2\cup \cdots\cup J_r = \{1,2,\ldots,k\}
\]
of the set~$\{1,\ldots,k\}$ so that~$F$ has the form
\[
  F(T_1,\ldots,T_k) = \sum_{i=1}^r c_i \prod_{j\in J_i} T_j.
\]
In other words,~$F$ is a linear combination of monomials in the
variables~$T_1,\ldots,T_k$ with the property that each variable
appears in exactly one monomial and to exactly the first power.
We also define the \emph{height of $F$} to be
\[
  h(F) := \max_{1\le i\le r} h(c_i).
\]
\end{definition}

\begin{theorem}
\label{thm:diffprod}  
Let $F(T_1,\ldots,T_k)\in\K[T_1,\ldots,T_k]$ be a multilinear
polynomial with split variables. Let $f(X)\in\K(X)$ be a rational
function of degree $d\ge2$.
\begin{parts}
\Part{(a)}
The set of $\a\in\ov\K$ for which there exists a $k$-tuple of \emph{distinct}
non-negative integers~$(n_1,n_2,\ldots,n_k)$ satisfying
\begin{equation}
  \label{eqn:Ffn1fn2fnk}
  F\left( f^{(n_1)}(\a), f^{(n_2)}(\a), \ldots, f^{(n_k)}(\a)) \right) = 0
\end{equation}
is a set of bounded height.
\Part{(b)}
If $d\ge3$, then for~$\a$ as in~\textup{(a)}, we have the explicit upper
bound
\[
h(\a) \le \frac{2k}{d^{k-1}}h(F) + \frac73 C_1(f) + \frac29\log2,
\]
where~$C_1(f)$ is the constant appearing in Lemma~\textup{\ref{lem:canht}(a)}.
\Part{(c)}
Let $\a\in\Wander_\K(f)$ have the property that $0\notin\Orbit_f(\a)$.
Then there are only finitely many $k$-tuples of integers
$n_1>n_2>\cdots>n_k\ge0$ satisfying~\eqref{eqn:Ffn1fn2fnk}, and there
is a bound for the number of such solutions that depends only
on~$\K$,~$f$ and~$F$, independent of~$\a$.
\end{parts}
\end{theorem}

\subsection{Multiplicative dependence and Zsigmondy-type results}
Many of our results on multiplicative independence would follow from a
sufficiently strong dynamical Zsigmondy theorem on primitive divisors,
but despite considerable attention in recent years, there are no
general unconditional result of the type that we would need.

We briefly expand on this remark.  Let~$f(X)$ be a rational function,
and let~$\a$ be a wandering point for~$f$.  We recall that a
valuation~$v$ is a \emph{primitive divisor of~$f^{(n)}(\a)$} if
$v\bigl(f^{(n)}(\a)\bigr)\ne0$ and $v\bigl(f^{(m)}(\a)\bigr)=0$ for
all $m<n$. The \emph{dynamical Zsigmondy set} associated to~$(f,\a)$
is the set of~$n$ such that~$f^{(n)}(\a)$ does not have a primitive
divisor.  With approrpiate conditions on~$f$ and~$\a$ to rule out
trivial counterexamples, it is conjectured that the dynamical
Zsigmondy set of~$(f,\a)$ is always finite.

There are a few unconditional results. Among them we
mention~\cite{FaGr,IngSilv}, which prove Zsigmondy finiteness when~$0$
is a preperiodic point, and~\cite{DoHa,Krieg}, which prove Zsigmondy
finiteness for unicritical binomials $f(X) = X^d + c$
over~$\Q$. Unfortunately, results such as our
Theorem~\ref{thm:fnkaufnars} require that~$0$ not be periodic, so the
former is not too helpful, and althugh the latter can be used in the
context of Theorem~\ref{thm:fnkaufnars}, it applies only to a very
restricted class of polynoimals.

The only known general results on finiteness of the dynamical
Zsigmondy set are conditional on very strong and difficult conjectures
such as the~$ABC$ conjecture or Vojta's conjecture; see, for
example,~\cite{BrTu,GhNgTu,GrNgTu,Silv13}.

It is also interesting to recall that for slower growing sequences,
finiteness of the Zsigmondy set may fail. For example, the Zsigmondy
set of a collection of polynomial values $\left\{f(n):n\in\ZZ\right\}$
has infinite Zsigmondy set; see~\cite{EvHa}.

%%%%%%%%%%%%%%%%%%%%%%%%%%%%%%%%%%%%%%%%%%%%%%%%%%%%%%%%%%%%%%%%%%%%%%
\section{Preliminaries}
%%%%%%%%%%%%%%%%%%%%%%%%%%%%%%%%%%%%%%%%%%%%%%%%%%%%%%%%%%%%%%%%%%%%%%
\subsection{Background on heights and valuations of iterations}
%%%%%%%%%%%%%%%%%%%%%%%%%%%%%%%%%%%%%%%%%%%%%%%%%%%%%%%%%%%%%%%%%%%%%%
In this section we collect some useful facts from Diophantine geometry
and arithmetic dynamics. We begin by recalling some standard
properties of the canonical height function associated to~$f$.

\begin{lemma}
\label{lem:canht}
There exists a unique function
\[
  \hhat_f : \PP^1(\overline\KK) \longrightarrow [0,\infty),
\] 
with the following  properties\textup:
\begin{parts}
\Part{(a)}
There is a constant $C_1(f)$ such that 
\[
  -C_1(f) \le \hhat_f(\g) - h(\g) \le C_1(f)\quad\text{for all $\g\in\PP^1(\overline\KK)$.}
\]
\Part{(b)}
For all $\g\in\PP^1(\overline\KK)$ we have,
$\hhat_f\left(f(\g)\right) = d\hhat_f(\g)$.
\Part{(c)}
For all $\g\in\PP^1(\overline\KK)$, we have   $\hhat_f(\g)=0$
if and only if $\gamma\in \PrePer(f)$.
\Part{(d)} 
There is a strict inequality
\[
C_2(\KK,f) := \inf\left\{ \hhat_f(\g): \g \in \Wander_\K(f) \right\} > 0.
\]
\end{parts}
\end{lemma}

\begin{proof}
For standard properties of dynamical canonical height functions,
including proofs of~(a,b,c), see, for
example,~\cite{CallSilv} or~\cite[Section~3.4]{Silv07}.
We mention that~(a) and~(b) suffice to determine~$\hhat_f$ uniquely.
For~(d), we first note from~(a) that there is an inclusion
\[
  \left\{ \g\in\PP^1(\KK) : \hhat_f(\g)\le 1 \right\}
  \subseteq \left\{ \g\in\PP^1(\KK) : h(\g)\le 1+C_1(f) \right\}.
\]
Next we use the fact that~$\PP^1(\KK)$ contains only finitely many
points of bounded height~\cite[Theorem~3.7]{Silv07}. Hence we may
take~$C_2(\KK,f)$ to be the smaller of~$1$ and the minimum
of~$\hhat_f(\g)$ over the finite set of points~$\g\in\PP^1(\KK)$
having infinite orbit and height~$h$ bounded by~$1+C_1(f)$, where~(c)
ensures that each of these finitely many values is strictly positive.
\end{proof}

\begin{definition}
The function $\hhat_f$ described in Lemma~$\ref{lem:canht}$ is called
the \emph{canonical height}.
\end{definition}

We  make frequent use of the following elementary consequence of
Lemma~\ref{lem:canht}.

\begin{lemma}
\label{lem:fnaeqc}
Let $\gamma\in\PP^1(\K)$. Then the set
\[
  \bigl\{ (n,\alpha)\in\ZZ_{\ge0}\times\Wander_\K(f) : f^{(n)}(\alpha)=\gamma \bigr\}
\]
is finite.
\end{lemma}
\begin{proof}
Let $(n,\alpha)$ be in the indicated set. Then
\begin{align*}
  h(\gamma) &= h\bigl(f^{(n)}(\alpha)\bigr) 
  &&\text{since $f^{(n)}(\alpha)=\g$,} \\
  &\ge \hhat_f\bigl(f^{(n)}(\alpha)\bigr) -C_1(f)
  &&\text{from Lemma~\ref{lem:canht}(a)} \\
  &= d^n\hhat_f(\alpha) -C_1(f)
  &&\text{from Lemma~\ref{lem:canht}(b)} \\
  &\ge d^nC_2(\K,f)-C_1(f)
  &&\text{from Lemma~\ref{lem:canht}(d)}.
\end{align*}
Thus $d^n \le \bigl(h(\gamma)+C_1(f)\bigr)C_2(\K,f)^{-1}$ is bounded, since
Lemma~\ref{lem:canht}(d) says that~$C_2(\K,f)>0$. Hence~$n$ is bounded. 
But for any given~$n$, we have
\[
h(\gamma) \ge d^n\hhat_f(\alpha)-C_1(f) \ge d^n\bigl(h(\alpha)-C_1(f)\bigr)-C_1(f),
\]
so $h(\alpha)\le d^{-n}\bigl(h(\gamma)+C_1(f)\bigr)+C_1(f)$ is bounded, which implies
that there are only finitely many possible values for~$\alpha$.
\end{proof}

\begin{definition}
We recall that $\b\in\PP^1(\ov \Q)$ is an \emph{exceptional point}
for~$f$ if its backward orbit
$\Orbit_f^-(\b):=\left\{\g\in\PP^1(\ov\Q):\b\in\Orbit_f(\g)\right\}$
is finite.
\end{definition}
  
It is a standard fact, see for example~\cite[Theorem~1.6]{Silv07},
that
\[
\#\Orbit_f^-(\b)\in\{1,2,\infty\},
\]
and that, after a change of coordinates, the cases
$\#\Orbit_f^-(\b)=1$ and $\#\Orbit_f^-(\b)=2$ correspond,
respectively, to $f(X)\in\KK[X]$ and $f(X)=cX^{\pm d}$.

The key to proving Theorem~\ref{thm:fnkaufnars} is the following
dynamical Diophantine approximation result, whose proof ultimately
relies on a suitably quantified version of Roth's theorem.

\begin{lemma}
\label{lem:lvfnageehffna}
Let $\a\in\Wander_\K(f)$.
Assume that~$0$ is not an exceptional point for~$f$. 
Let $S$ be a finite set of places of~$\KK$, and let $1\ge\e>0$.
Then there is a constant $C_3(\KK,S,f,\e)$ such that
\begin{align*}
\max \left\{ n\in\ZZ_{\ge0}:\sum_{v\in S}
  \logplus\left(\left \|f^n(\a)\right \|_v^{-1}\right) \ge \e \hhat_f\left(f^n(\a)\right)
  \right\}&\\
  \le C_3(\KK,S&,f,\e).
\end {align*}
\end{lemma}

\begin{proof}
The finiteness of the indicated set has originally been proven
in~\cite{Silv93} without an explicit upper
bound. The quantified version that we quote
here is~\cite[Theorem~11(c)]{HsiSil}.  More precisely, we
apply~\cite[Theorem~11(c)]{HsiSil} with $A=0$ and $P=\alpha$.  We
note that for~$A=0$ and with our normalization of the $v$-adic
absolute values, the local distance function~$\l_v(Q,A)$
in~\cite{HsiSil} is given by $\l_v(\b,0)=\log^+\|\b\|_v^{-1}$.
\end{proof}

 \subsection{The genus of plane curves of the form $F(X)=cG(X)Y^m$}
There is a well-known formula for the genus of a superelliptic curve
$f(X) =Y^m$, where $f(X)\in \CC[X]$ is a polynomial; see for 
example~\cite[Exercise A.4.6]{HinSil}.  In this section we find a similar
formula in the case that $f(X)\in\CC(X)$ is a rational function,
modulo certain restrictions on~$m$.  We use square brackets
$[\star,\ldots,\star]$ to denote points in projective space and
parentheses $(\star,\ldots,\star)$ to denote points in affine space.

\begin{definition}\label{def:nu} For a polynomial $F(X) \in\CC[X]$, we define $\nu(F)$
  to be the number of distinct complex roots of~$F$, i.e., the number
  of roots counted \emph{without multiplicity}.
\end{definition}

We frequently apply Definition~\ref{def:nu}  to a product of two relatively primes
polynomials $F(X),G(X)\in\CC[X]$, in which case   
\[
  \nu(FG)=\#f^{-1}\bigl(\{0,\infty\}\bigr)
\]
is equal to the total number of poles and zeros of the rational
function $f(X)=F(X)/G(X)\in\CC(X)$.

\begin{lemma} 
\label{lem:FXcGXYm}
Let $F(X),G(X)\in\CC[X]$ be non-zero polynomials with no common roots,
and assume that they are not both constant.
Let $d_F =\deg F$ and $d_G=\deg G$,   
and let $m$ be an integer satisfying
\begin{equation}
  \label{eqn:gcdmd1}
  m \ge d_F+2\quad\text{and}\quad
  \gcd(m,d_F!\cdot d_G!)=1.
\end{equation}
Let~$C$ be the affine curve
\[
  C : F(X) = G(X)Y^m,
\]
and let~$\widetilde C$ be a smooth projective model of~$C$. 
\begin{parts}
\Part{(a)}
The curve $\widetilde C$ is irreducible.
\Part{(b)} 
The genus of $\widetilde C$ is
given by the formula
\[
  \genus(\widetilde C) = \begin{cases}
     \frac12(\nu(FG)-1)(m-1) &\text{if $d_F\ne d_G$,} \\
     \frac12(\nu(FG)-2)(m-1)   &\text{if $d_F= d_G$.}  
  \end{cases}
\]
\end{parts}
\end{lemma}

\begin{proof} 
(a)\enspace
We need to prove that the polynomial $F(X)-G(X)Y^m$ does not factor in
$\CC[X,Y]$.  We apply~\cite[Chapter~VI, Theorem~9.1]{LangAlg}
to the polynomial $Y^m-F(X)/G(X)$ in the variable~$Y$ with
coefficients in the field~$\CC(X)$.  Since~$m$ is odd, we see that
$Y^m-F(X)/G(X)$ is irreducible in $\CC(X)[Y]$ provided for every
prime~$p\mid m$, the rational function $F(X)/G(X)$ is not a $p$th
power in~$\CC(X)$. But our choice of~$m$ ensures that~$p\nmid d_Fd_G$,
so~$F(X)/G(X)$ cannot be a $p$th power.
\par\noindent(b)\enspace
Write
\[
  F(X)=a\prod_{i=1}^{r} (X-\a_i)^{e_i}
  \quad\text{and}\quad
  G(X)=b\prod_{i=1}^{s} (X-\b_i)^{\e_i},
\]
so $\nu(FG)=r+s$.  We have assumed that $m>d_F$, so the Zariski
closure~$\bar C$ of~$C$ in~$\PP^2$ is given by the homogeneous
equation
\[
  \bar C : a\prod_{i=1}^{r} (X-\a_iZ)^{e_i} Z^{m+d_G-d_F} =
  cb\prod_{i=1}^{s} (X-\b_iZ)^{\e_i} Y^m.
\]  
An elementary calculation, which we give in Appendix~\ref{app:A},
shows that the singular points of $\bar C$ in~$\PP^2$ are:
\[
  [1,0,0]\quad\text{always,}\qquad
  [\a_i,0,1]\quad\text{if $e_i\ge2$,}\qquad
  [0,1,0]\quad\text{if $\deg G\ne1$.}
\]

We consider the map $\widetilde C\to\bar C$.  In an infinitesimal
neighbourhood of the singular the equation of~$\bar C$ looks like
$Z^{m+d_G-d_F} = Y^m$, so there are $\gcd(d_G-d_F,m)$ points
of~$\widetilde C$ lying over~$[1,0,0]$. This gives two cases.  If
$d_F=d_G$, then there are~$m$ points of~$\widetilde C$ lying
over~$[1,0,0]$.  On the other hand, if $d_F\ne d_G$, then $1\le
|d_F-d_G|\le\max\{d_F,d_G\}$, so~\eqref{eqn:gcdmd1} implies that
$\gcd(d_G-d_F,m)=1$, and thus in this case there is~$1$ point
of~$\widetilde C$ lying over~$[1,0,0]$.

Similarly, we note that in an infinitesimal neighbourhood of a singular point
$[\a_i,0,1]$, the equation of~$C$ looks locally like
$(X-\a_i)^{e_i}=Y^m$. Blowing up this singularity yields $\gcd(e_i,m)$
points on~$\widetilde C$ lying over~$[\a_i,0,1]$, and our
assumption~\eqref{eqn:gcdmd1} implies that $\gcd(e_i,m)=1$.

Finally, if $[0,1,0]$ is singular, then dehomogenising $Y=1$
gives an affine  equation for~$\bar C$ of the form
\[
  a\prod_{i=1}^{r} (X-\a_iZ)^{e_i} Z^{m+d_G-d_F}
  =  cb\prod_{i=1}^{s} (X-\b_iZ)^{\e_i}.
\]  
We blow up the point $(0,0)$. We start with the chart on the blowup
given by $X=SZ$. Substituting and canceling the common factor
of~$Z^{d_G}$, we obtain the equation
\begin{equation}
  \label{eqn:chart1}
  a\prod_{i=1}^{r} (S-\a_i)^{e_i} Z^{m}
  =  cb\prod_{i=1}^{s} (S-\b_i)^{\e_i}.
\end{equation}  
The points on the blowup above $(X,Z)=(0,0)$ are the points
$(S,Z)=(\b_i,0)$.  The point~$(\b_i,0)$ is singular if and only
if~$\e_i\ge2$, but just as in our earlier calculation, our choice
of~$m$ ensures that there is only one point of~$\widetilde C$ lying above
each of these singular points.  Thus we have found~$s$ points
of~$\widetilde C$ lying above the singular point $[0,1,0]\in\bar C$.  (We
note that this is also true if $[0,1,0]$ is nonsingular, since that
case occurs if $d_G=1$, which implies also that~$s=1$.)

It remains to check the other chart on the blowup, which is given
by $Z=TX$. Substituting and canceling a power of~$X$ yields
\[
  a\prod_{i=1}^{r} (1-\a_iT)^{e_i} T^{m+d_G-d_F} X^{m}
  -  cb\prod_{i=1}^{s} (1-\b_iT)^{\e_i}  = 0.
\]  
The only point in this chart that is not in the other chart is
$(T,X)=(0,0)$, which is not a point on the blowup of the curve. So we
obtain no further points in this chart.

To summarize, if we let $\pi:\widetilde C\to\bar C$ denote the map coming from
the various blowups used to desingularize~$\bar C$, we have
\begin{align*}
  \#\pi^{-1}\left([1,0,0]\right) &= \begin{cases}
    1&\text{if $d_F\ne d_G$,} \\
    m&\text{if $d_F=d_G$,} \\
  \end{cases} \\
  \#\pi^{-1}\left([\a_i,0,1]\right) &= 1\quad\text{for each $1\le i\le r$,} \\
  \#\pi^{-1}\left([0,1,0]\right) &= s.
\end{align*}

We now consider the covering map
\[
  \f : \widetilde C \longrightarrow \bar C
  \xrightarrow[ {[X,Y,Z]}\mapsto{[X,Z]} ] {} \PP^1,
\]
where we note that $\deg \f=m$. Indeed, the map~$\f$ is Galois with
Galois group~$\ZZ/m\ZZ$. However, we need to be a bit careful, since
the map from $\bar C\to\PP^1$ is not defined at $[0,1,0]\in\bar C$,
although it is defined at the~$s$ points of~$\widetilde C$ lying
over~$[0,1,0]$.  More precisely, on the chart~\eqref{eqn:chart1}
for~$\widetilde C$ with affine coordinates $(S,Z)$ and $X=SZ$, the
map~$\phi$ is given by $\f(S,Z)=[X,Z]=[SZ,Z]=[S,1]$, so the~$s$ points
$\widetilde\b_1,\ldots,\widetilde\b_s$ on~$\widetilde C$ lying
over~$[0,1,0]\in\bar C$ satisfy~$\f(\widetilde\b_i)=[\b_i,1]$, and
indeed we have $\f^{-1}\bigl([\b_i,1]\bigr\}=\{\b_i\}$ for each $1\le
i\le s$.

We use the Riemann--Hurwitz formula, see~\cite[Theorem~1.5]{Silv07},
\[
  2\genus(\widetilde C) - 2 = \left(2\genus(\PP^1)-2\right)\cdot\deg \f
  +  \sum_{P\in\widetilde C} \left(e_\f(P) - 1 \right).
\]
Substituting $\deg \f=m$ and $\genus(\PP^1)=0$, and applying~\cite[Corollary~1.3]{Silv07}, this yields
\begin{align*}
  2\genus(\widetilde C)
  &= 2(1-m) + \sum_{P\in\widetilde C} \left(e_\f(P) - 1 \right) \\
  &=2(1-m) + \sum_{Q\in\PP^1}\sum_{P\in \f^{-1}(Q)} \left(e_\f(P) - 1 \right)\\
  &= 2(1-m)+\sum_{Q\in\PP^1} \left( \deg \f - \#\f^{-1}(Q) \right) \\
  &= 2(1-m)+ \sum_{i=1}^r \left( m - \#\f^{-1}\bigl([\a_i,1]\bigr) \right)
   \\
  &\qquad\quad {} + \sum_{i=1}^s \bigl( m - \#\f^{-1}([\b_i,1]) \bigr) + \left( m - \#\f^{-1}\bigl([1,0]\bigr) \right)\\
  &= 2(1-m)+r(m-1) + s(m-1) + 
  \begin{cases}
    m-1&\text{if $d_F\ne d_G$} \\
    0&\text{if $d_F= d_G$} \\    
  \end{cases} \\
  &= \begin{cases}
     (r+s-1)(m-1) &\text{if $d_F\ne d_G$,} \\
     (r+s-2)(m-1)   &\text{if $d_F= d_G$.} 
  \end{cases}
\end{align*}  
This completes the proof.
\end{proof}

\subsection{Generalised Schinzel-Tijdeman theorem}
We also need the following general version of the Schinzel-Tijdeman
Theorem~\cite{SchTij}, which also
extends~\cite[Theorem~2.3]{BEG}. More precisely, the constant $C$ in
Lemma~\ref{eq:GenS-T} stated below depends only on the prime ideal
divisors of the coefficient $b$, and not on its height as
in~\cite{BEG}. For our purposes, this improvement is crucial.

For a set of places~$S$ of~$\K$, we write $R_S$ for the ring of
$S$-integers and $R_S^*$ for the group of $S$-units.

\begin{lemma}
\label{eq:GenS-T}
Let $\K$ be a number field, and let $S$ be a finite set of places of
$\K$ containing all infinite places.  Let $f\in R_S[X]$ be a polynomial
without multiple roots and of degree at least $2$. There is an
effectively computable constant $C(f, \K, S)$, depending only on $f$,
$\K$ and $S$, so that the following holds\textup: If~$b\in R_S^*$ and
if the equation
\[
  f(x)=b \cdot y^m
\]
has a solution satisfying
\[
  x,y,\in R_S\quad\text{and}\quad y\notin\{0\}\cup R_S^*,
\]
then
\[
  m \leq C(f, \K, S).
\]
\end{lemma}

\begin{proof}
The proof is nearly identical to that of~\cite[Theorem~2.3]{BEG}, with
$\hat{h}$ replaced by $h(f)$. Here we only indicate the slight changes
which are needed in the proof, under the assumption that
$b\in{R_S^*}$. Thus in this proof we use the notation from the proof
of~\cite[Theorem~2.3]{BEG}, other than sticking with our notation for
the sets of $S$-integers and $S$-units.

In~\cite[Lemmas~4.15 and~4.17]{BEG}, $\hat{h}$ may be clearly replaced
by $h(f)$. Further, the estimates
in~\cite[Equations~(5.1)--(5.10)]{BEG} remai valid if we replace
$\hat{h}$ by $h(f)$, under the assumption that $b\in R_S^*$.

Instead of~\cite[Equations~(5.11) and~(5.12)]{BEG}, we argue as
follows. We may assume without loss of generality that
\[
X\geq \max(C_3, m(4d)^{-1}(\log 3d)^{-3}),
\]
with $C_3$ being the constant specified in~\cite[Equation~(5.12)]{BEG}. Indeed, if
\[
X\geq \max(C_3, m(4d)^{-1}(\log 3d)^{-3})
\]
then by 
\[
\frac{1}{[\K:\Q]}N_S(y^m)=\frac{1}{[\K:\Q]}N_S(b y^m)\leq h(b y^m) \leq h(f(x)),
\]
we obtain
\[
m \leq [\K:\Q] \frac{nX+h(a_0)}{\log 2}.
\]

The rest of the proof of~\cite[Theorem~2.3]{BEG} follows without any
changes other than replacing $\hat{h}$ by $h(f)$ at each occurrence.
\end{proof}

\begin{remark} 
It is crucial for our argument that the constant~$C(f,\K,S)$ in Lemma~\ref{eq:GenS-T}
is independent of~$b\in R_S^*$.
\end{remark}

\section{Proofs of main results}

\subsection{Preliminary discussion} 
In this section we give the proofs of
Theorems~\ref{theorem:fnainG},~\ref{thm:fnkaufnars},~\ref{thm:fnkaufna},~\ref{thm:2multdep-arb}
and~\ref{thm:diffprod}, which are main results of this paper. Throughout these proofs we use the following 
common notation. 
We let~$S$ be the following set of absolute values on~$\KK$:
\begin{equation}
\label{eqn:S}
  S := M_\K^\infty \cup \bigl\{v\in M_\K^0 : \text{$\|\g\|_v\ne1$ for some $\g\in\Gamma$} \bigr\}.
\end{equation}
The set~$S$ is finite, since~$\Gamma$ is finitely generated, so in
proving our main theorems, we may assume that~$\Gamma$ is the full
group of $S$-units,
\[
   \Gamma =  R_S^* = \left\{ \b\in\KK^*: 
  \text{$\|\b\|_v=1$ for all $v\in M_\K \setminus S$} \right\}.
  \]
This is convenient because~$R_S^*$ is multiplicatively saturated
in~$\K^*$, i.e., if~$\g\in\K^*$ and~$\g^n\in R_S^*$ for some~$n\ne0$,
then~$\g\in R_S^*$.

\subsection{Proof of Theorem~$\ref{theorem:fnainG}$}
(a)\enspace
As noted earlier, this statement has originally been proven
in~\cite[Proposition~1.5 (a)]{KLSTYZ}, but as a convenience to the
reader, we include the short proof.  By assumption, the set
\begin{equation}
\label{eqn:ainKfainG}
  \Gcal(\K,f,\Gamma):=
  \bigl\{ \alpha\in\K : f(\a) \in \Gamma \bigr\}
\end{equation}
is infinite.  Let $m=\max\{d!+1,5\}$, We fix coset representatives
$c_1,\ldots,c_t\in R_S^*$ for the finite group $R_S^*/(R_S^*)^m$.
Then for every~$\alpha\in\Gcal(\K,f,\Gamma)$, we can find an
index~$i(\alpha)\in\{1,\ldots,t\}$ and some~$u\in R_S^*$ so that
$f(\a)=c_{i(\alpha)}u^m$.  Hence the assumption that the
set~$\Gcal(\K,f,\Gamma)$ is infinite means that we can find a
$c\in{R_S^*}$ so that the set
\[
  \bigl\{ (\alpha,u)\in\K\times R_S^* : f(\alpha)=cu^m \bigr\}
\]
is infinite. It follows that every~$(\alpha,u)$ in this set is on 
the algebraic curve
\[
  C : f(X)=cY^m,\quad\text{and hence that}\quad \#C(\K)=\infty.
\]
Writing $f(X)=F(X)/G(X)$ with $F(X),G(X)\in\K[X]$ relatively prime,
the curve~$C$ has the equation
$F(X)=cY^mG(X)$, so Lemma~\ref{lem:FXcGXYm} tells us that the genus of
a smooth projective model~$\tilde C$ for~$C$ is given by
\[
  \genus(\tilde C) = 
  \frac{m-1}{2}\left(
  \nu(FG) - \left\{\begin{tabular}{@{}ll@{}}
  1&if $\deg(F)\ne\deg(G)$\\
  2&if $\deg(F)=\deg(G)$
  \end{tabular}
  \right\}
  \right),
\]
where we recall that $\nu(H)$ denotes the number of \emph{distinct}
complex roots of a polynomial $H$.  Faltings'
theorem~\cite{Falt83,Falt86} (Mordell conjecture) tells us
that~$C(\K)$ is finite if~$\genus(\tilde C)\ge2$, so the fact
that~$\#C(\K)=\infty$ implies that one of the following cases is true.
\begin{align*}
\textbf{Case I}: && \nu(FG)&=1, & \deg(F) &\ne \deg(G). \\
\textbf{Case II}: && \nu(FG)&=2, & \deg(F) &= \deg(G). 
\end{align*}

In Case~I, the fact that~$\nu(FG)=1$ means that one of~$F$ or~$G$ is
constant and the other has only a single root. Since~$\deg f=d$, this
means that~$f$ has the form~$f(X)=a(x-b)^{\pm d}$. In Case~II, the
fact that~$\nu(FG)=2$ and~$F$ and~$G$ both have degree~$d$ implies
that they each have a single root, so $F(X)=a_1(X-b)^d$ and
$G(X)=a_2(X-c)^d$, which proves that~$f$ has the
form~$f(X)=a(X-b)^d/(X-c)^d$.
\par\noindent(b)\enspace
We are assuming that the set
\begin{equation}
\label{eqn:naZn2WKffnaG}
  \Fcal_2(\K,f,\Gamma) =
  \bigl\{ (n,\alpha)\in\ZZ_{\ge2}\times\Wander_\K(f) : f^{(n)}(\a) \in \Gamma \bigr\}
\end{equation}
is infinite. Since every pair~$(n,\alpha)$ in $\Fcal_2(\K,f,\Gamma)$
has~$n\ge2$, we have a well-defined map
\[
  \Fcal_2(\K,f,\Gamma) \longrightarrow \Gcal(\K,f^{(2)},\Gamma),\quad
  (n,\alpha) \longmapsto f^{(n-2)}(\alpha).
\]
If the set~$\Gcal(\K,f^{(2)},\Gamma)$ is finite, then there are only
finitely many possible values for~$f^{(n-2)}(\alpha)$ as~$(n,\alpha)$
range over the set~$\Fcal_2(\K,f,\Gamma)$, and then
Lemma~\ref{lem:fnaeqc} tells us that there are only finitely many
possiblities for~$n$ and~$\alpha$, contradicting the assumption
that~$\#\Fcal_2(\K,f,\Gamma)=\infty$.

Hence we must have $\#\Gcal(\K,f^{(2)},\Gamma)=\infty$, and then~(a) applied to
the map~$f^{(2)}(X)$ tells us that~$f^{(2)}(X)$ has one of the following three forms:
\begin{align*}
  f^{(2)}(X) &= \f_{\pm}(X) :=  a(X-b)^{\pm d^2}, \\
  f^{(2)}(X) &= \psi(X) :=  a(X-b)^{d^2}/(X-c)^{d^2}.
\end{align*}
We claim that this forces~$f(X)$ to have the form~$aX^{\pm d}$. 

We observe that~$\f_{\pm}(X)$ is totally ramified at~$b$ and~$\infty$,
so in this case~$f$ is totally ramified at~$b$,~$\infty$,~$f(b)$,
and~$f(\infty)$. Similarly, the map~$\psi(X)$ is totally ramified
at~$b$ and~$c$, so in this case~$f$ is totally ramified
at~$b$,~$c$,~$f(b)$, and~$f(c)$.  The Riemann--Hurwitz
formula~\cite[Theorem~1.5]{Silv07} implies that a rational map has at
most two points where it is totally ramified, and hence
\begin{align*}
  f^{(2)}(X)&=\f_{\pm}(X) &\Longrightarrow&& f(b),f(\infty)&\in\{b,\infty\}, \\
  f^{(2)}(X)&=\psi(X) &\Longrightarrow&& f(b),f(c)&\in\{b,c\}.
\end{align*} 
This leads to several subcases, which are detailed in
Figure~\ref{figure:fphipsicases}, in which we use the symbol
`$\rightarrow\!\leftarrow$' to indicate a contradiction.
Looking at Figure~\ref{figure:fphipsicases}, we see that if
$f^{(2)}=\f_{+}$ and $f(b)=b$, then~$0$ and~$\infty$ are totally
ramified fixed points of~$f$, so~$f(X)=aX^d$, while if
$f^{(2)}=\f_{+}$ and $f(b)=\infty$, then~$0$ and~$\infty$ are totally
ramified and form a $2$-cycle, so~$f(X)=aX^{-d}$. Finally, all cases
with~$f^{(2)}=\f_{-}$ and~$f^{(2)}=\psi$ lead to contradictions. This
completes the proof of Theorem~\ref{theorem:fnainG}. \qed
\begin{figure}[p]
\begin{alignat*}{5}
\cline{1-5}
f^{(2)} &= \f_{+},\; f(b)=b
&\;\Longrightarrow
&  \quad
  \left(
  \begin{array}{@{}c@{}c@{}c@{}c@{}c@{}}
    b &\xrightarrow{f}& b &\xrightarrow{f}& 0 \\
    \infty &\xrightarrow{f}& f(\infty) &\xrightarrow{f}& \infty \\
  \end{array}
  \right)\\
&&\;\Longrightarrow 
& \quad \left\{ \begin{tabular}{@{}l@{}}
    $b = f(b) = 0$, so\\
    $f(\infty)\ne b$, so\\
    $f(\infty) = \infty$.
  \end{tabular} \right.
  \\
%%%%%%%%%%%%%%%%%%%%%%%%%%%%%%%%%%%%%%%%
\cline{1-5}
f^{(2)} &= \f_{+},\; f(b)=\infty
&\;\Longrightarrow
&  \quad
  \left(
  \begin{array}{@{}c@{}c@{}c@{}c@{}c@{}}
    b &\xrightarrow{f}& \infty &\xrightarrow{f}& 0 \\
    \infty &\xrightarrow{f}& f(\infty) &\xrightarrow{f}& \infty \\
  \end{array}
  \right)\\
&&\;\Longrightarrow 
&   \quad \left\{ \begin{tabular}{@{}l@{}}
    $0 = f^{(2)}(b) = f(\infty)$, so\\
    $\infty=f^{(2)}(\infty)=f(0)$ and\\
    $b=f^{-1}(\infty)=0$.\\
  \end{tabular} \right.
\\
%%%%%%%%%%%%%%%%%%%%%%%%%%%%%%%%%%%%%%%%
\cline{1-5}
f^{(2)} &= \f_{-},\; f(b)=b
&\;\Longrightarrow
&  \quad
  \left(
  \begin{array}{@{}c@{}c@{}c@{}c@{}c@{}}
    b &\xrightarrow{f}& b &\xrightarrow{f}& \infty \\
    \infty &\xrightarrow{f}& f(\infty) &\xrightarrow{f}& 0 \\
  \end{array}
  \right)\\
&&\;\Longrightarrow
&  \quad  \infty = f(b) = b.\; \rightarrow\!\leftarrow
\\
%%%%%%%%%%%%%%%%%%%%%%%%%%%%%%%%%%%%%%%%
\cline{1-5}
f^{(2)} &= \f_{-},\; f(b)=\infty
&\;\Longrightarrow
&  \quad
  \left(
  \begin{array}{@{}c@{}c@{}c@{}c@{}c@{}}
    b &\xrightarrow{f}& \infty &\xrightarrow{f}& \infty \\
    \infty &\xrightarrow{f}& f(\infty) &\xrightarrow{f}& 0 \\
  \end{array}
  \right)\\
&&\;\Longrightarrow
&  \quad \infty = f(\infty) =  0.\; \rightarrow\!\leftarrow
\\
%%%%%%%%%%%%%%%%%%%%%%%%%%%%%%%%%%%%%%%%
\cline{1-5}
f^{(2)} &= \psi,\; f(b)=b
&\;\Longrightarrow
&  \quad
  \left(
  \begin{array}{@{}c@{}c@{}c@{}c@{}c@{}}
    b &\xrightarrow{f}& b &\xrightarrow{f}& 0 \\
    c &\xrightarrow{f}& f(c) &\xrightarrow{f}& \infty \\
  \end{array}
  \right)\\
&&\;\Longrightarrow
&  \quad \left\{ \begin{tabular}{@{}l@{}}
    $0 = f(b) = b$, so\\
    $f(c)\ne b$, so $f(c)=c$, so\\
    $\infty = f(c) = c.\; \rightarrow\!\leftarrow$
  \end{tabular} \right.
\\
%%%%%%%%%%%%%%%%%%%%%%%%%%%%%%%%%%%%%%%%
\cline{1-5}
f^{(2)} &= \psi,\; f(b)=c
&\;\Longrightarrow
& \quad
  \left(
  \begin{array}{@{}c@{}c@{}c@{}c@{}c@{}}
    b &\xrightarrow{f}& c &\xrightarrow{f}& 0 \\
    c &\xrightarrow{f}& f(c) &\xrightarrow{f}& \infty \\
  \end{array}
  \right)\\
&&\;\Longrightarrow\;
&  \quad \left\{\begin{tabular}{@{}l@{}}
    $f(c)\ne c$ (else $0=f(c)=\infty$),\\
    so $f(c)=b$, so\\
    $\infty = f(b) = c.\; \rightarrow\!\leftarrow$
  \end{tabular} \right. \\
%%%%%%%%%%%%%%%%%%%%%%%%%%%%%%%%%%%%%%%%
\cline{1-5}
\end{alignat*}
\caption{Case-by-case analysis for values of $f^{(2)}$ and $f(b)$}
\label{figure:fphipsicases}
\end{figure}

\subsection{Proof of Theorem~$\ref{thm:fnkaufnars}$}
We are going to prove the contrapositve statement:
\[
  \#f^{-1}\bigl(\{0,\infty\}\bigr)\ge3
  \quad\Longrightarrow\quad
  \text{$\Ecal_\rho(\K,f,\Gamma,r,s)$ is a finite set.}
\]
The assumption that~$f$ has at least~$3$ zeros and poles implies in
particular that $f$ does not have the form $cX^{\pm d}$, so we know
that at least one of~$0$ or~$\infty$ is not exceptional.  If~$0$ is
exceptional, then we claim that we can swap~$0$ and~$\infty$ and
replace~$f(X)$ with the polynomial $g(X)=f(X^{-1})^{-1}\in\K[X]$. To
see this, we note that an initial wandering point~$\a$ for~$f$, which
necessarily satisfies~$\a\ne0$ since~$0$ is exceptional, is mapped to
the initial  wandering point~$\a^{-1}$ for~$g$.  Further,
$f^{(n)}(\a) = g^{(n)}(\a^{-1})^{-1}$, so
\begin{align*}
  f^{(n+k)}(\a)^r/f^{(k)}(\a)^{s}\in\Gamma
  &\quad\Longleftrightarrow\quad
  g^{(k)}(\a^{-1})^{s}/g^{(n+k)}(\a^{-1})^r\in\Gamma\\* 
  &\quad\Longleftrightarrow\quad
  g^{(n+k)}(\a^{-1})^r/g^{(k)}(\a^{-1})^{s}\in\Gamma,
\end{align*}
where the first implication follows from the definition of~$g$, and the
second from the fact that~$\Gamma$ is a group. 
Thus if~$0$ is exceptional for~$f$, then there is a bijection
\begin{align*}
  \Ecal_\rho(\K,f(X),\Gamma,r,s) &\longrightarrow
  \Ecal_\rho(\K,f(X^{-1})^{-1},\Gamma,r,s), \\
  (n,k,\alpha,u) &\longmapsto (n,k,\alpha^{-1},u^{-1}).
\end{align*}
Hence, without loss of generality, we may assume that~$0$ is  not an
exceptional point for~$f$, which in turn allows us to apply
Lemma~\ref{lem:lvfnageehffna}.

We want to study the set  of triples
$(n,k,\a) \in \ZZ_{\ge1} \times \ZZ_{\ge0} \times \Wander_\K(f)$ such that
\begin{equation}
  \label{eqn:fnkvafnavvnotS}
  \left\| f^{(n+k)}(\a) \right\|_v^r  = \left\| f^{(k)}(\a) \right\|_v^{s}
  \quad\text{for all $v\in M_\K \setminus S$.}
\end{equation}

For an arbitrary choice of~$\varepsilon$, to be specified later, we
let $C_3(\KK,S,f,\varepsilon)$ be the constant from
Lemma~\ref{lem:lvfnageehffna}, and we split the proof into two
cases, depending on the size of $n+k$.

\Case{1}{$n+k\ge C_3(\KK,S,f,\varepsilon)$}
In this case Lemma~\ref{lem:lvfnageehffna} tells us that $(n,k,\alpha)$
satisfies
\begin{equation}
  \label{eqn:lpfnkleehfnk}
  \sum_{v\in S}
  \logplus\left(\left\|f^{(n+k)}(\a)\right\|_v^{-1}\right) \le \varepsilon  \hhat_f \left(f^{(n+k)}(\a)\right).
\end{equation}

Since $h(\g)=h(\g^{-1})$ and using~\eqref{eqn:defht}, we compute
\begin{align*}
  h\left(f^{(n+k)}(\a)\right)
 &= h\left(f^{(n+k)}(\a)^{-1}\right)  
 = \sum_{v\in M_\K} \logplus\left(\left\|f^{(n+k)}(\a)\right\|_v^{-1}\right)\\
   &= \sum_{v\in S} \logplus\left(\left\|f^{(n+k)}(\a)\right\|_v^{-1}\right)\\
 & \qquad \qquad \qquad  +  \sum_{v\in M_\K \setminus S} \logplus\left(\left\|f^{(n+k)}(\a)\right\|_v^{-1}\right).
\end{align*}
Now, using~\eqref{eqn:fnkvafnavvnotS} and~\eqref{eqn:lpfnkleehfnk}
\begin{align*}
  h\left(f^{(n+k)}(\a)\right) 
   &\le \varepsilon  \hhat_f \left(f^{(n+k)}(\a)\right) +  \sum_{v\in M_\K \setminus S}\logplus\left(\left\|f^{(k)}(\a)\right\|_v^{-s/r}\right)
 \\
 &\le \varepsilon  \hhat_f \left(f^{(n+k)}(\a)\right) +  \frac{|s|}{|r|}  h\left(f^{(k)}(\a)^{-1}\right) \\
 &= \varepsilon  \hhat_f \left(f^{(n+k)}(\a)\right) +  d^{\rho-1} h\left(f^{(k)}(\a)\right),
\end{align*}
where for the last equality we have used the fact that~$\rho$ is
defined by the relation $|s|/|r|=d^{\rho-1}$.  We next use
Lemma~\ref{lem:canht}(a) to replace the Weil height~$h$ with the
canonical height~$\hhat_f$. This yields
\[
  \hhat_f\left(f^{(n+k)}(\a)\right) -  C_1(f)  
 \le   \varepsilon  \hhat_f \bigl(f^{(n+k)}(\a)\bigr) +  d^{\rho-1}\bigl(  \hhat_f \bigl(f^{(k)}(\a)\bigr) + C_1(f)  \bigr) ,
\]
and a little algebra leads to
\[
 (1-   \varepsilon)\hhat_f\bigl(f^{(n+k)}(\a)\bigr)  \le  d^{\rho-1} \hhat_f \bigl(f^{(k)}(\a)\bigr)+ C_1(f)\left(1+d^{\rho-1}\right).
\]
Using Lemma~\ref{lem:canht}(b) gives
\[
 (1-   \varepsilon) d^{n+k}\hhat_f\bigl(\a\bigr)  \le  d^{\rho-1+k}\hhat_f \bigl(\a\bigr) + C_1(f)\left(1+d^{\rho-1}\right),
\]
and hence
\begin{equation}
  \label{eq:dkdnfaC1}
d^k  \left((1-   \varepsilon) d^{n} - d^{\rho-1}\right) \hhat_f\bigl(\a\bigr)  \le C_1(f)\left(1+d^{\rho-1}\right).
\end{equation}

Taking $\varepsilon = 1/3$,  we see that for  $n \ge \rho $ we have
\[
  (1-   \varepsilon) d^{n} - d^{\rho-1}=
 \frac{2}{3}d^{n}- d^{\rho-1} \ge
  \frac{4}{3}d^{n-1}- d^{\rho-1} \ge \frac{1}{3}d^{n-1},
\]
and thus we derive from~\eqref{eq:dkdnfaC1} that 
\begin{equation}
\label{eqn:2C1fgedkhfa}
d^{n+k-1}  \hhat_f(\a) \le 3  C_1(f)\left(1+d^{\rho-1}\right).
\end{equation}
Using $d^{n+k-1} \hhat_f=\hhat_f\circ f^{(n+k-1)}$ and again
Lemma~\ref{lem:canht}(a), which gives $\hhat_f\ge h-C_1(f)$,
\eqref{eqn:2C1fgedkhfa} yields
\[
  h\bigl(f^{(n+k-1)}(\a)\bigr) \le 3  C_1(f)\left(1+d^{\rho-1}\right)+C_1(f).
\]
Thus $f^{(n+k-1)}(\a)$ is in a set of bounded height, where the bound
depends only on~$f$ and~$\rho$, so there are only finitely many
possible values for $f^{(n+k-1)}(\a)$. Then Lemma~\ref{lem:fnaeqc}
tells us that there are only finitely many possible values
for~$n$,~$k$, and~$\alpha$.  This completes the proof in Case~1 that
there only finitely many $4$-tuples $(n,k,\a,u)$ satisfying
$n+k\ge{C}_3\bigl(\KK,S,f,\frac13\bigr)$. We note
that for this case we only needed the assumption that $f$ is of degree
$d\ge 2$ and does not have the form $cX^{\pm d}$.

\Case{2}{$n+k< C_3\bigl(\KK,S,f,\frac 13\bigr)$} 
Replacing~$r$ and~$s$ with~$-r$ and~$-s$ if necessary, we may assume
that $r>0$. Further, since~$n+k$ is assumed bounded,
we may assume that~$k$ and~$n$ are fixed, and we need to show that
there are only finitely many $\alpha\in\P^1(\K)$ satisfying
\begin{equation}
\label{eq:fnk}
  f^{(n+k)}(\alpha)^{r}/{f^{(k)}}(\alpha)^{s} \in R_S^*.
\end{equation}

By assumption we have $s\ne 0$. We let
\[
  g(X)=f^{(n)}(X)^r/X^{s}\in\K(X).
\]
Then~\eqref{eq:fnk} implies that
\[
  g\left({f^{(k)}}(\alpha)\right) = f^{(n+k)}(\alpha)^{r}/{f^{(k)}}(\alpha)^{s}  \in R_S^*.
\]
In the notation of Theorem~\ref{theorem:fnainG}(a), this says that
${f^{(k)}}(\alpha)\in\Gcal(\K,g,R_S^*)$.  If $\Gcal(\K,g,R_S^*)$
is finite, then ${f^{(k)}}(\alpha)$ takes on only finitely many
values, and Lemma~\ref{lem:fnaeqc} tells us that there are only
finitely many values of~$\alpha$.

On the other hand, if $\Gcal(\K,g,R_S^*)$ is infinite, then
Theorem~\ref{theorem:fnainG}(a) says that~$g(X)$ has at most two zeros
and poles. But the assumption that~$0$ is not periodic implies
that~$0$ is a pole of~$g$, and the assumption that~$f(X)\ne
cX^{\pm{d}}$ implies that~$f^{(n)}(X)$ has at least two poles or zeros
distinct from~$0$.  Hence~$g$ has at least three poles and zeros. \qed

\subsection{Proof of Theorem~$\ref{thm:fnkaufna}$}
We have shown in the proof of Theorem~\ref{thm:fnkaufnars} that
if 
\[
  \bigl\{(n,k,\a,u) \in \Ecal(\KK,f,\Gamma) : n+k > C_3(\KK,S,f,1/3) \bigr\}
\]
is infinite, then $f(X) = cX^{\pm d}$ for some $c \in \K^*$.  We are
thus reduced to the situation that the following three conditions hold:
\begin{itemize}
\setlength{\itemsep}{0pt}
\item
The function $f(X)\in\K(X)$ is not of the form~$cX^{\pm d}$ for any $c \in \K^*$.
\item
The integers $n\ge1$ and $k\ge0$ are fixed and satisfy
\[
  n+k \le C_3(\KK,S,f,1/3). 
\]
\item
There are infinitely many $\a \in \Wander_\K(f)$ satisfying the
equation~\eqref{eqn:fnkaufka}, i.e., there are infinitely 
many pairs
\[
  (\a,u)\in\Wander_\KK(f)\times R_S^*\quad\text{satisfying}\quad
  f^{(n+k)}(\a) = u f^{(k)}(\a).
\]
\end{itemize}
Since~$k$ is fixed, we can set $\b=f^{(k)}(\a)$, so we see that there
are infinitely many pairs
\begin{equation}
  \label{eqn:buKRSFbub}
  (\b,u)\in\Wander_\KK(f)\times R_S^*\quad\text{satisfying}\quad
  f^{(n)}(\b) = u \b.
\end{equation}

We define 
\[
  m  = \LCM(2,3,\ldots,d^n+1) + 1,
\]
so in particular
\[
  m\ge7 \quad\text{and}\quad \gcd\bigl(m,(d^n+1)!\bigr) = 1.
\]
  
We fix coset representatives $c_1,\ldots,c_t\in R_S^*$ for the finite
group $R_S^*/(R_S^*)^m$. Then each~$u\in R_S^*$ can be written in the
form $c_i\g^m$ for some $i=i(u)\in\{1,\ldots,t\}$ and some
$\g=\g(u)\in R_S^*$. Hence in order to determine the number of pairs
satisfying~\eqref{eqn:buKRSFbub}, it suffices to study, for each
fixed~$c\in R_S^*$, the number of pairs  
\begin{equation}
  \label{eqn:buKRSFbcgmb}
  (\b,\g)\in\Wander_\KK(f)\times R_S^*\quad\text{satisfying}\quad
  f^{(n)}(\b) = c \g^m \b.
\end{equation}
We note that $\b\ne0$, since otherwise~$\b$ is periodic.

In order to apply Lemma~\ref{lem:FXcGXYm}, we consider three possible
ways to write the rational function~$f^{(n)}(X)$, depending on its
order of vanishing at~$X=0$.  More precisely, we choose
$F_n,G_n\in\KK[X]$ satisfying $\gcd(F_n,G_n)=1$ so that~$f^{(n)}(X)$
has one of the following forms:
\begin{align*}
  \textbf{Case A:}&& f^{(n)}(X) &= {F_n(X)}/{G_n(X)},& F_n(0)&\ne0, \\
  \textbf{Case B:}&& f^{(n)}(X) &= {XF_n(X)}/{G_n(X)},& F_n(0)G_n(0)&\ne0, \\
  \textbf{Case C:}&& f^{(n)}(X) &= {X^eF_n(X)}/{G_n(X)},& F_n(0)G_n(0)&\ne0,\quad e\ge2.
\end{align*}
(We note that by homegeneity, we may assume that either~$F_n$ or~$G_n$ is monic.)
These three cases lead to studying points on the following three curves:
\begin{align*}
  \textbf{Case A:}&& C:F_n(X)&=cY^mXG_n(X), \\
  \textbf{Case B:}&& C:F_n(X)&=cY^mG_n(X), \\
  \textbf{Case C:}&& C:X^{e-1}F_n(X)&=cY^mG_n(X).
\end{align*}
We let $\widetilde C$ denote a smooth projective model of~$C$. 
Applying Lemma~\ref{lem:FXcGXYm}, each of the three cases
leads to two subcases, according to the relative degrees of~$F_n$
and~$G_n$. To ease notation, we let 
\[
  d_{F_n}=\deg F_n \quad\text{and}\quad d_{G_n}=\deg G_n.
\]
Then Lemma~\ref{lem:FXcGXYm} yields:
\[
\frac{2\genus(\widetilde C)}{m-1} = \begin{cases}
  \nu(F_nG_n)
  &\text{in Case A, $d_{F_n}\ne d_{G_n}+1$,} \\
  \nu(F_nG_n)-1
  &\text{in Case A, $d_{F_n}= d_{G_n}+1$,} \\
  \nu(F_nG_n)-1
  &\text{in Case B, $d_{F_n}\ne d_{G_n}$,} \\
  \nu(F_nG_n)-2
  &\text{in Case B, $d_{F_n}= d_{G_n}$,} \\
  \nu(F_nG_n)
  &\text{in Case C, $d_{F_n}+e-1\ne d_{G_n}$,} \\
  \nu(F_nG_n)-1
  &\text{in Case C, $d_{F_n}+e-1= d_{G_n}$,} 
\end{cases}
\]
where, as before, $\nu(H)$ be the number of \emph{distinct} complex
roots of a polynomial $H$.  
Since $m\ge7$, we see that either $\genus(\widetilde C)=0$, or else
$\genus(\widetilde C)\ge3$. In the latter case, Faltings'
theorem~\cite{Falt83, Falt86} (Mordell conjecture) tells us
that~$C(\KK)$ is finite. So it remains to analyze the six cases with
$\genus(\widetilde{C})=0$, as described in Table~\ref{table:genus0}.
The remainder of the proof is a case-by-case analysis.
\begin{table}[ht]
\[
\begin{array}{|c||c|c|c|} \hline
  \text{Case} & C & d_{F_n}~\text{and}~d_{G_n} & \nu(F_nG_n) \\ \hline\hline
  \textbf{A1} & F_n(X)=cY^mXG_n(X)& d_{F_n}\ne d_{G_n}+1& 0 \\ \hline
  \textbf{A2} & F_n(X)=cY^mXG_n(X)& d_{F_n}=d_{G_n}+1& 1 \\ \hline
  \textbf{B1} & F_n(X)=cY^mG_n(X)& d_{F_n}\ne d_{G_n}& 1 \\ \hline
  \textbf{B2} & F_n(X)=cY^mG_n(X)& d_{F_n}= d_{G_n}& 2 \\ \hline
  \textbf{C1} & X^{e-1}F_n(X)=cY^mG_n(X)& d_{F_n}+e+1\ne d_{G_n}& 0\\ \hline
  \textbf{C2} & X^{e-1}F_n(X)=cY^mG_n(X)& d_{F_n}+e+1=d_{G_n}& 1\\ \hline
\end{array}
\]  
\caption{Cases with $\genus(\widetilde C)=0$}
\label{table:genus0}
\end{table}

\Case{A1}{$\nu(F_nG_n)=0$}
In this case $F_n$ and~$G_n$ are constants, so $f^{(n)}(X)$ is
constant, contradicting $\deg f \ge2$.

\Case{C1}{$\nu(F_nG_n)=0$}
Again $F_n$ and~$G_n$ are constants, so $f^{(n)}(X)=cX^e$ for some
$e\ge2$. This is only possible if~$f$ has the form~$f(X)=aX^{\pm d}$.

\Case{A2}{$\nu(F_nG_n)=1$ \& $d_{F_n}=d_{G_n}+1$}
One of~$F_n$ or~$G_n$ is constant. The degree condition forces~$G_n$
to be constant, and hence~$d_{F_n}=1$. Therefore
\[
  d^n = \deg f^{(n)}(X) = \deg F_n(X)/G_n(X) =\max\{d_{F_n},d_{G_n}\} = 1.
\]
This  contradicts $\deg f\ge2$.

\Case{C2}{$\nu(F_nG_n)=1$ \& $d_{F_n}+e-1=d_{G_n}$}
One of~$F_n$ or~$G_n$ is constant. The degree condition forces~$F_n$
to be constant, and hence~$d_{G_n}=e-1$ and thus $d^n = \deg f^{(n)} = e$.
Therefore $f^{(n)}(X) = X^eF_n(X)/G_n(X)$ has the form
\[
  f^{(n)}(X) = aX^{d^n}/(X-b)^{d^n-1}.
\]
We are going to prove that this forces~$n=1$.

More generally, 
let
\[
\psi(X)  = aX^D/(X-b)^{D-1}
  \quad\text{with $D\ge2$ and $ab\ne0$,}
\]
and suppose that~$g$ and~$h$ are rational functions satisfying
\[
  g\circ h(X) = \psi(X).
\]
We claim that~$g$ or~$h$ has degree~$1$. 

To see this, we first note that 
\begin{equation}
\label{eq:efb}
  e_\psi(\infty)=1\quad\text{and}\quad e_\psi(b)=D-1,
\end{equation}
since if we use the linear fractional transformation $L(X)=X^{-1}$ to
move~$\infty$ to~$0$ and~$b$ to~$b^{-1}$, we have
\[
  L^{-1}\circ \psi\circ L(X) =\psi(X^{-1})^{-1}  =  a^{-1}X  (1 -bX)^{D-1},
\]
from which it is easy to read off the ramification indices.  We next
use the fact that
\[
  h^{-1}\left(g^{-1}(\infty)\right) = (g\circ h)^{-1}(\infty) =\psi^{-1}(\infty) = \{b,\infty\}
\]
to conclude that~$\#g^{-1}(\infty)$ equals~$1$ or~$2$.

Suppose first that $g^{-1}(\infty)=\{c\}$ consists of a single point. 
We use multiplicativity of ramificaiton indices and~\eqref{eq:efb} to compute
\[
  1 = e_\psi(\infty) = e_h(\infty)e_g(c)
  \quad\text{and}\quad
  D-1 = e_\psi(b) = e_h(b)e_g(c).
\]
The first equality gives~$e_g(c)=1$, and then the second gives~$e_h(b)=D-1$. 
This gives the estimate
\[
  \deg h \ge e_h(b) = D-1 = \deg\psi-1 = (\deg g)(\deg h) - 1.
\]
Therefore
\[
  1 \ge (\deg h)(\deg g-1),
\]
which forces either $\deg g=1$ or $\deg h=1$. 

We suppose next that $g^{-1}(\infty)$ consists of two points, i.e., 
suppose that~$h(\infty)\ne h(b)$. Then $h^{-1}\bigl(h(b)\bigr)=\{b\}$,
so $e_h(b)=\deg h$. Further, from
\[
  1 = e_\psi(\infty) = e_h(\infty)e_g\bigl(h(\infty)\bigr),
\]
we see that $e_g\bigl(h(\infty)\bigr)=1$, and hence using
\[
\deg g = \sum_{c\in g^{-1}(\infty)} e_g(c) = e_g\bigl(h(\infty)\bigr)
+ e_g\bigl(h(b)\bigr),
\]
we find that $e_g\bigl(h(b)\bigr)=\deg g-1$. 
Hence
\[
  D-1 = e_\psi(b)  = e_h(b)e_g\bigl(h(b)\bigr) 
  = (\deg h)(\deg g-1) = D -\deg h.
\]
Therefore $\deg h=1$.

This completes the proof in Case~C2 that~$n=1$, and hence that $f(X)=aX^d/(X-b)^{d-1}$.  

\Case{B1}{$\nu(F_nG_n)=1$ \& $d_{F_n}\ne d_{G_n}$}
The assumption that~$\nu(F_nG_n)=1$ implies that one of~$F_n$
and~$G_n$ is constant and the other has exactly one root.
Further, since in Case~B we have $f^{(n)}(X)=XF_n(X)/G_n(X)$, we see
that~$f^{(n)}$ has one of the following forms:
\[
  f^{(n)}(X) = aX(X-b)^{d^n-1}
  \quad\text{or}\quad
  f^{(n)}(X) = aX/(X-b)^{d^n}.
\]
In order to prove that~$n=1$, we suppose that~$g$ and~$h$ are rational
functions satisfying
\[
  g\circ h = aX(X-b)^{D-1} \quad\text{or}\quad
  g\circ h = aX(X-b)^{-D}
\]
with
\[
b\ne 0 \quad\text{and}\quad D\ge2.
\]
Our goal is to show that either~$g$ or~$h$ is linear.

We start with $g\circ h = aX(X-b)^{D-1}$, so $(g\circ h)^{-1}(0)=\{0,b\}$.
Thus~$\#g^{-1}(0)=1$ or~$2$. Suppose first that $\#g^{-1}(0)=1$, say
$g^{-1}(0)=\{c\}$. Then~$g$ is totally ramified at~$c$, i.e., $e_g(c)=\deg g$.
The form of~$g\circ h$ implies that it is unramified at~$0$, so, by~\eqref{eqn:mult e} we have
\[
  1 = e_{g\circ h}(0) = e_h(0)e_g\left(h(0)\right) = e_h(0)e_g(c)=e_h(0)\deg g.
\]
Hence~$\deg g=1$.

Similarly, if $\#g^{-1}(0)=2$, say $g^{-1}(0)=\{c_1,c_2\}$, then
possibly after relabeling, we have $h^{-1}(c_1)=\{0\}$ and
$h^{-1}(c_2)=\{b\}$. In particular, the map~$h$ is totally ramified
at~$0$, i.e., $e_h(0)=\deg h$.  Again using the fact that~$g\circ h$
is unramified at~$0$, using~\eqref{eqn:mult e}, we find that
\[
  1 = e_{g\circ h}(0) = e_h(0)e_g\left(h(0)\right) = e_h(0)e_g(c_1)= e_g(c_1) \deg h.
\]
Hence~$\deg h=1$.

We next do the case that $g\circ h = aX(X-b)^{-D}$, so $(g\circ h)^{-1}(0)=\{0,\infty\}$.
Thus~$\#g^{-1}(0)=1$ or~$2$. Suppose first that $\#g^{-1}(0)=1$, say
$g^{-1}(0)=\{c\}$. Then~$g$ is totally ramified at~$c$, i.e., $e_g(c)=\deg g$.
The form of~$g\circ h$ implies that it is unramified at~$0$, so by~\eqref{eqn:mult e}
\[
  1 = e_{g\circ h}(0) = e_h(0)e_g\bigl(h(0)\bigr) = e_h(0)e_g(c)=e_h(0)\deg g.
\]
Hence~$\deg g=1$.

Similarly, if $\#g^{-1}(0)=2$, say $g^{-1}(0)=\{c_1,c_2\}$, then
possibly after relabeling, we have $h^{-1}(c_1)=\{0\}$ and
$h^{-1}(c_2)=\{\infty\}$. In particular, the map~$h$ is totally ramified
at~$0$, i.e., $e_h(0)=\deg h$.  Again using the fact that~$g\circ h$
is unramified at~$0$, using~\eqref{eqn:mult e}, we find that
\[
  1 = e_{g\circ h}(0) = e_h(0)e_g\left(h(0)\right) = e_h(0)e_g(c_1)= e_g(c_1) \deg h.
\]
Hence~$\deg h=1$.

This completes the proof for Case~B1 that $n=1$, and
either~$f(X)=aX(X-b)^{d-1}$ or $f(X)=aX/(X-b)^d$.

\Case{B2}{$\nu(F_nG_n)=2$ \& $d_{F_n}=d_{G_n}$}
The fact that~$F_n$ and~$G_n$ have the same degree means that neither
can be constant, so $\nu(F_nG_n)=2$ implies that~$F_n$ and~$G_n$ each
have exactly one root. Further, since $f^{(n)}(X)=XF_n(X)/G_n(X)$ in Case~B,
we can read off the degrees of~$F_n$ and~$G_n$, so we find that
\[
  f^{(n)}(X) = \frac{aX(X-b)^{d^n-1}}{(X-c)^{d^n-1}}
  \quad\text{with $b,c,0$ distinct.}
\]

In order to prove that~$n=1$, we suppose that~$g$ and~$h$ are rational
functions satisfying
\[
  g\circ h = \frac{aX(X-b)^D}{(X-c)^D}
  \quad\text{with $b,c,0$ distinct and $D\ge1$,}
\]
and our goal is to show that either~$g$ or~$h$ is linear.
We have $(g\circ h)^{-1}(\infty)=\{c,\infty\}$, so~$\#g^{-1}(\infty)=1$
or~$2$.

Suppose first that~$\#g^{-1}(\infty)=1$, say~$g^{-1}(\infty)=\{\widetilde{c}\}$.
Then~$g$ is totally ramified at~$\widetilde{c}$, i.e.,~$e_g(\widetilde{c})=\deg g$. The
form of~$g\circ h$ implies that~$g\circ h$ is unramified at~$\infty$,
hence, by~\eqref{eqn:mult e}
\[
  1=e_{g\circ h}(\infty) = e_h(\infty)e_g\left(h(\infty)\right)
  =e_h(\infty)e_g(\widetilde{c})=e_h(\infty) \deg g.
\]
Therefore $\deg g=1$.

Similarly, if~$\#g^{-1}(\infty)=2$, say~$g^{-1}(\infty)=\{c_1,c_2\}$,
then possibly after relabeling,~$h^{-1}(c_1)=\{\infty\}$ and
$h^{-1}(c_2)=\{c\}$. In particular,~$h$ is totally ramified
at~$\infty$, so~$e_h(\infty)=\deg h$. Again using the fact
that~$g\circ h$ is unramified at~$\infty$, using~\eqref{eqn:mult e}, we find that
\[
  1=e_{g\circ h}(\infty) = e_h(\infty)e_g\left(h(\infty)\right)
  =e_h(\infty)e_g(c_1)= e_g(c_1) \deg h.
\]
Therefore $\deg h=1$.

This completes the proof in Case~B2 that~$n=1$,
and hence that $f(X)=aX(X-b)^{d-1}/(X-c)^{d-1}$. \qed

\subsection{Proof of Theorem~$\ref{thm:2multdep-arb}$} 
Let $\alpha \in \K$ have the property that there exists a pair of
non-negative integers $(m,n)$ with $m>n>0$ and integers $r$ and $s$
and an $S$-unit $u\in R_S^*$ such that
\begin{equation}
\label{eq:multrs}
  \bigl(f^{(m)}(\alpha)\bigr)^r = u \bigl(f^{(n)}(\alpha)\bigr)^s.
\end{equation} 

Since, by Northcott's Theorem, the set $\PrePer(f)\cap \K$ is finite, it is enough to prove the finiteness of the set of points $\alpha\in\Wander_\K(f)$ satisfying~\eqref{eq:multrs}.

We note that if $r=0$ or $s=0$, then the finiteness of $\alpha\in\K$
satisfying~\eqref{eq:multrs} follows directly from
Theorem~\ref{theorem:fnainG} since
$f$ has $d \ge 2$ distinct roots. Thus, we may assume from now on that $rs\ne0$.

From~\eqref{eq:multrs} and the power saturation of~$R_S^*$ in~$\K^*$, we see
that
\[
  u\in R_S^*\cap(\K^*)^{\gcd(r,s)} = (R_S^*)^{\gcd(r,s)}.
\]
This allows us to take the~$\gcd(r,s)$-root of~\eqref{eq:multrs}, so
without loss of generality we may assume that
\[
  \gcd(r,s)=1. 
\]

For a prime ideal $\fp$ of the ring of integers $R_\K$ of $\K$, we
denote the (normalized additive) valuation on $\K$ at the place
corresponding to $\fp$ by $v_\fp:\K^*\onto\ZZ$.  As usual, we say that
a polynomial~$f(X)=c_0+c_1X+\cdots+c_dX^d$ has bad reduction at~$\fp$
if either $v_\fp(c_i)<0$ for some~$i$ or if~$v_\fp(c_d)>0$; otherwise
we say it has good reduction.  We let 
\[
  \cS_{f,\Gamma} :=  S \cup 
   \{\gp\in M_\K^0 : \text{$f$ has bad reduction at $\gp$} \},
\]
where we recall that~$S$ is the set of places~\eqref{eqn:S} with~$\Gamma=R_S^*$. 
It is a standard fact, even for rational functions, that if~$f$ has
good reduction at a prime, then so do all of its iterates;
see~\cite[Proposition~2.18(b)]{Silv07}. (This is especially easy to
see for polynomials.) Hence
\begin{equation}
\label{eq:SmS}
  \cS_{f^{(m)},\Gamma} \subseteq \cS_{f,\Gamma}, \quad \text{for all $m\ge1$.}
\end{equation}

We let
\[
  \RSfG := \bigl\{ \vartheta \in \K : \text{$ v_\fp(\vartheta)\ge 0$ for all $\fp \not \in \cS_{f,\Gamma}$} \bigr\}
\]
be the ring of $\cS_{f,\Gamma}$-integers in $\K$, and $ \RSfG^*$
denotes the group of $\cS_{f,\Gamma}$-units in $\K$.

Recall that we assume $rs\ne 0$. Replacing~$r,s$ by~$-r,-s$ if
necessary, we may assume that~$r>0$.  We consider two cases depending
on whether the initial point $\alpha$ in~\eqref{eq:multrs} satisfies
$\alpha \in \RSfG$.

\Case{A}{$\alpha \in  \RSfG$} 
In this case, our definitions ensure that every iterate
$f^{(k)}(\alpha)$ is in $\RSfG$, i.e.,
\begin{equation}
\label{eq:vp}
  v_\fp\bigl(f^{(k)}(\alpha)\bigr)\ge 0 
  \quad\text{for all $k\ge0$ and all $\fp\not \in \cS_{f,\Gamma}$.}
\end{equation}

We distinguish now the following four  subcases:

\Case{A.1}{$r>0$, $s<0$} 
Then equation~\eqref{eq:multrs} becomes
\[
\bigl(f^{(m)}(\alpha)\bigr)^r \bigl(f^{(n)}(\alpha)\bigr)^t= {u}\quad\text{with  $t=-s>0$.}
\]
Using the fact that $u\in\RSfG^*$, it follows that
\[
rv_\fp\bigl(f^{(m)}(\alpha)\bigr)+tv_\fp\bigl(f^{(n)}(\alpha)\bigr)=0
\quad\text{for all $\fp\not\in\cS_{f,\Gamma}$.}
\]
Since $r,t>0$, and since~\eqref{eq:vp} tells us that the iterates have
non-negative valuation, we conclude that
\[
  v_\fp\bigl(f^{(m)}(\alpha)\bigr)=v_\fp\bigl(f^{(n)}(\alpha)\bigr)=0
  \quad\text{for all $\fp\not\in\cS_{f,\Gamma}$.}
\]
In particular, we have $f^{(m)}(\alpha)\in\RSfG^*$.  Since $m>n>0$, we
have $m\ge2$, so Theorem~\ref{theorem:fnainG}(b) tells us that there
are only finitely many~$(m,\alpha)$ with this property unless~$f$ has
the form~$f(X)=aX^{\pm d}$, which would contradict the assumption
that~$0$ is not periodic for~$f$.

\Case{A.2}{$r> 0$ and   $s\ge 2$}   
Since
$\gcd(r,s)=1$, we can choose integers $a$ and $b$ with
$ar+bs=1$. Then~\eqref{eq:multrs} becomes
\begin{equation}
\label{eq:s pow}
  f^{(m)}(\alpha) ={u}^a \Bigl(\bigl(f^{(n)}(\alpha)\bigr)^a
  \bigl(f^{(m)}(\alpha)\bigr)^b\Bigr)^s.
\end{equation}
Since $f^{(m)}(\alpha)\in\RSfG$ and ${u}\in\RSfG^*$, we  see from~\eqref{eq:s pow} that
\[
  \bigl(f^{(n)}(\alpha)\bigr)^a \bigl(f^{(m)}(\alpha)\bigr)^b \in  \RSfG. 
\]

Clearly, if
\[
\bigl(f^{(n)}(\alpha)\bigr)^a \bigl(f^{(m)}(\alpha)\bigr)^b\in \RSfG^*
\]
then $ f^{(m)}(\alpha) \in \RSfG^*$ as well. 
As in Case A.1, since  $m\ge2$,  Theorem~\ref{theorem:fnainG}(b) tells us that there
are only finitely many~$(m,\alpha)$ with this property unless~$f$ has
the form~$f(X)=aX^{\pm d}$, which would contradict the assumption
that~$0$ is not periodic for~$f$.
Thus we may assume that
\begin{equation}
\label{eq:avoid RS*}
\bigl(f^{(n)}(\alpha)\bigr)^a \bigl( f^{(m)}(\alpha)\bigr)^b\not\in
\RSfG^*.
\end{equation}

Writing 
\[
f^{(m)}(\alpha)  = f\bigl(f^{(m-1)}(\alpha) \bigr),
\]
we see from Lemma~\ref{eq:GenS-T} that under the condition~\eqref{eq:avoid RS*}
 the exponent $s \ge 2$ in~\eqref{eq:avoid RS*} is bounded above by a quantity depending only on $\K$, $f$ and $\cS_{f,\Gamma}$. 

We assume first that $\deg f\ge 3$. Applying~\cite[Theorem~2.2]{BEG}
to \eqref{eq:s pow}, we conclude that there are only finitely many
values for $f^{(m-1)}(\alpha)$, and thus, Lemma~\ref{lem:fnaeqc} says
that there are only finitely many possibilities for~$m$ and~$\alpha\in\K$.

Similarly, if $\deg f=2$ and $s\ge 3$, then~\cite[Theorems~2.1]{BEG}
again says that there are only finitely many values for
$f^{(m-1)}(\alpha)$, so we are done.

It remains to deal with the case that $\deg f=s=2$. Since $m\ge 2$, writing 
\[
f^{(m)}(\alpha)  = f^{(2)}\bigl(f^{(m-2)}(\alpha) \bigr),
\]
and applying~\cite[Theorem~2.2]{BEG} to \eqref{eq:s pow} with
$f^{(2)}$ (since we assume that $f^{(2)}$ has only simple roots), we conclude that there are only finitely many values for
$f^{(m-2)}(\alpha)$, and therefore, by Lemma~\ref{lem:fnaeqc},  finitely many possibilities for~$m$ and~$\alpha\in\K$.

\Case{A.3}{$r\ge 2$ and   $s=1$} 
We note that if $n\ge 2$, then the same discussion as above holds for this case too (where we replace $m$ be $n$ and $s$ by $r$). We therefore consider only the case $n=1$, and~\eqref{eq:multrs} becomes 
\begin{equation}
\label{eq:r 2}
f(\alpha)=u^{-1}  \bigl(f^{(m)}(\alpha)\bigr)^r.
\end{equation}

If $r\ge 3$, then~\cite[Theorems~2.1]{BEG}
again says that there are only finitely many values for $\alpha$ and 
$f^{(m)}(\alpha)$, so we are done.

We assume thus $r=2$ in~\eqref{eq:r 2}. 
Since $\deg f\ge 2$ and $f$ has only simple roots, we apply Theorem~\ref{thm:fnkaufnars}, where $\rho$ in this case satisfies $\rho<1$, to conclude that there are finitely many $m$ and $\alpha$ satisfying~\eqref{eq:r 2}.

\Case{A.4}{$r=s=1$} 
This case leads to an equation of the form~\eqref{eqn:fnkaufka},  which  is covered by Theorem~\ref{thm:fnkaufnars}.

\Case{B}{$\alpha \not \in  \RSfG$} 
We choose a prime ideal $\fq$ of $R_\K$ with 
\[
\fq\not \in \cS_{f,\Gamma} \quad\text{and}\quad v_\fq(\alpha) < 0.
\]
Then, using~\eqref{eq:SmS}, we see that $\fq\not \in
\cS_{f^{(k)},\Gamma}$ for all $k\ge1$, and thus from the proof
of~\cite[Theorem~4.11]{OSSZ1}, we have
\begin{equation}
\label{eq:ord fm}
   v_\fq\bigl(f^{(k)}(\alpha)\bigr) = d^k  v_\fq(\alpha)\quad\text{for all $k\ge0$.}
\end{equation}
(In dynamical terms, this formula reflects the fact that~$\alpha$ is
in the $\fq$-adic attracting basin of the superattracting fixed
point~$\infty$ of~$f$.)

Applying this to~\eqref{eq:multrs}, we conclude that 
\begin{equation}
\label{eqn:rdmvqa}
   r d^m  v_\fq(\alpha) =  s d^n  v_\fq(\alpha).
\end{equation}
We also know that
\begin{equation}
\label{eqn:r0mnrsgcd}
  d\ge2,\quad r>0,\quad m>n>0,\quad \gcd(r,s)=1,\quad\text{and}\quad
  v_\fq(\alpha)\ne 0.
\end{equation}
It follows from~\eqref{eqn:rdmvqa} and~\eqref{eqn:r0mnrsgcd}
that
\[
  r=1\quad\text{and}\quad s=d^{m-n},
\]
and hence~\eqref{eq:multrs} becomes
\begin{equation}
\label{eq:multr1}
  f^{(m)}(\alpha) = {u}\bigl(f^{(n)}(\alpha)\bigr)^{d^{m-n}}. 
\end{equation} 

To facilitate a comparison with the proof of Theorem~\ref{thm:fnkaufnars}, 
we let~$m=n+k$, and then we reverse the 
roles of~$n$ and~$k$, which changes~\eqref{eq:multr1} into
\begin{equation}
\label{eq:multr2}
  f^{(n+k)}(\alpha) = {u}\bigl(f^{(k)}(\alpha)\bigr)^{d^{n}}. 
\end{equation} 

If  $n=1$, i.e., $f^{(k+1)}(\alpha) = {u}\bigl(f^{(k)}(\alpha)\bigr)^d$  then 
we use Theorem~\ref{theorem:fnainG}(a), applied with the rational function $f(X)/X^d$. 
Note that,  since $f$ has only simple roots and $X\nmid f$,   
the rational function $f(X)/X^d$  is not of  one of the forms 
described in Theorem~\ref{theorem:fnainG}(a).  Hence we have the desired finiteness in 
this case.
  
We now need to show that~\eqref{eq:multr2} has finitely many solutions
\begin{equation}
\label{eq:nkauZ2Z1}
  (n,k,\a,u) \in \ZZ_{\ge 2} \times \Z_{\ge 1} \times \Wander_\K(f) \times \Gamma.
\end{equation} 
We note that even for a fixed value of~$n$, we cannot apply
Theorem~\ref{thm:fnkaufnars} directly with $r=1$ and $s=d^n$, because
Theorem~\ref{thm:fnkaufnars} only deals with solutions satisfying $n \ge \rho$, 
while in our case 
\[
\rho =\log_d(s/r)+1=\log_d(d^n)+1=n+1.
\]

We replace~$\K$ by the extension field~$\L$ generated by the set of values
\[
  \{\b\in\ov\K : \b^{d^2}\in\Gamma\}.
\]
Note that~$\L/\K$ is a finite extension depending only on~$\K$,~$d$, and~$\Gamma$,
since~$\Gamma/\Gamma^{d^2}$ is a finite group.  

We consider now two cases: $d\ge 3$ and $d=2$.

If $d\ge 3$ we consider  the curve
\[
  C_1 : f(X) = Y^{d^2}.
\]
Since $f$ has only simple roots, applying the genus formula~\cite[Exercise A.4.6]{HinSil} for a smooth projective model $\widetilde{C}_1$ of the affine curve $C_1$, we have
\[
\genus(\widetilde{C}_1)=(d-1)(d^2-2)/2\ge 7
\] 
when $d\ge 3$. Therefore, by Faltings' theorem~\cite{Falt83,Falt86}, the set of~$\L$-rational
points~$C_1(\L)$ on~$C_1$ is finite. However, every solution $(n,k,\a,u)$
of the form~\eqref{eq:nkauZ2Z1} to the equation~\eqref{eq:multr2}
gives an $\L$-rational point~$C_1(\L)$ via the formula
\[
  \Bigl(  f^{(n+k-1)}(\alpha), u^{1/d^2}  \bigl(f^{(k)}(\alpha)\bigr)^{d^{n-2}} \Bigr) \in C_1(\L).
\]
(Note that $u^{1/{d^2}}\in\L$ by the definition of~$\L$.)
We conclude in particular that $f^{(n+k-1)}(\alpha)$ takes
on only finitely many values, and then Lemma~\ref{lem:fnaeqc} says
that there are only finitely many possibilities for~$n$,~$k$,
and~$\alpha$.  

If $d=2$, we consider  the curve
\[
  C_2 : f^{(2)}(X) = Y^{d^2}.
\]
In this case, since $f^{(2)}$ has simple roots by hypothesis, the formula~\cite[Exercise A.4.6]{HinSil} for the genus of a smooth projective model $\widetilde{C}_2$ of the curve $C_2$ becomes 
\[\genus(\widetilde{C}_2)=(d^2-1)(d^2-2)/2=3.
\] As above, Faltings' theorem~\cite{Falt83,Falt86} implies that the set of~$\L$-rational
points~$C_2(\L)$ on~$C_2$ is finite, and thus there are finitely many possibilities for~$n$,~$k$,
and~$\alpha$.

In fact, we remark that if $d=2$ and $n\ge 3$, one does not need the condition that $f^{(2)}$ has simple roots. Indeed, in this case one can consider  the curve
\[
  C_3 : f(X) = Y^{d^3}.
\]
Since $f$ has simple roots, the formula for the genus of a smooth projective model $\widetilde{C}_3$ of $C_3$ is 
\[\genus(\widetilde{C}_3)=(d-1)(d^3-2)/2=3,
\] and thus, by  Faltings' theorem~\cite{Falt83,Falt86}, we obtain the same finiteness conclusion as above.

This concludes the proof of Case~B,
and with it the proof of Theorem~\ref{thm:2multdep-arb}. \qed

\subsection{Proof of Theorem~$\ref{thm:diffprod}$} 

Relabeling, we assume that the $k$-tuple of distinct integers is ordered
so that
\[
  n_1>n_2>\cdots>n_k\ge0.
\]
By assumption, the polynomial~$F$ has the form
\[
  F(T_1,\ldots,T_k) = \sum_{i=1}^r c_i \prod_{j\in J_i} T_j
\]
for some disjoint partition $J_1\cup\cdots\cup J_r=\{1,2,\ldots,k\}$.
Relabeling, we may assume that $1\in J_1$.

We first note that $\bigl\{\a\in\ov\K : 0\in\Orbit_f(\a)\bigr\}$ is a
set of bounded height. Explicitly, if $f^{(n)}(\a)=0$, then
\begin{align*}
  h(\a) &\le \hhat_f(\a)+C_1(f) =  d^{-n} \hhat_f\bigl(f^{(n)}(\a)\bigr) +C_1(f) \\
 &  = d^{-n} \hhat_f(0) + C_1(f) \le \hhat_f(0) + C_1(f) \\
  & \le h(0) + 2C_1(f) = 2C_1(f).
\end{align*}
Thus for the proofs of~(a) and~(b), we may restrict attention
to~$\a\in\K$ such that $0\notin\Orbit_f(\a)$. Note that this also
implies that~$r\ge2$.

\par\noindent(a)\enspace
We rewrite~\eqref{eqn:Ffn1fn2fnk} to isolate the~$n_1$-term, 
\[
  f^{(n_1)}(\a) = \sum_{i=2}^r (-c_i) \prod_{j\in J_i} f^{(n_j)}(\a)
  \biggm/ c_1 \prod_{\hidewidth j\in J_1\setminus\{1\}\hidewidth } f^{(n_j)}(\a).
\]
(This is where we use the fact that $0\notin\Orbit_f(\a)$.)  We take
the height of both sides and use the submultiplicativity and
subadditivity properties
of~$h$~\cite[Exercise~B.20]{Silv07}, i.e.,
\[
  h\left(\prod_{i=1}^n\b_i\right) \le \sum_{i=1}^n h(\b_i),
  \qquad
  h\left(\sum_{i=1}^n\b_i\right) \le \sum_{i=1}^n h(\b_i) + \log n,
\]
together with the fact that $h(\beta^{-1}) = h(\beta)$ for $\b\ne0$.
This yields
\begin{align*}
h\left(f^{(n_1)}(\a)\right)
  &\le h\left(\sum_{i=2}^r (-c_i) \prod_{j\in J_i} f^{(n_j)}(\a)\right)
    + h\left(c_1 \prod_{\hidewidth j\in J_1\setminus\{1\}\hidewidth } f^{(n_j)}(\a)\right) \\
  &\le \sum_{i=2}^r \left( h(c_i) + \sum_{j\in J_i} h\left(f^{(n_j)}(\a)\right) \right)+ \log(r-1) \\
  &\omit\hfill$\displaystyle{}  + h(c_1) + \sum_{\hidewidth j\in J_1\setminus\{1\}\hidewidth } h\left(f^{(n_j)}(\a)\right)$ \\
  &= \sum_{j=2}^k  h\left(f^{(n_j)}(\a)\right) + \sum_{i=1}^r h(c_i) + \log(r-1) .
\end{align*}
Using Lemma~\ref{lem:canht}(a) to switch to canonical heights gives
\begin{align*}
  \hhat_f\left(f^{(n_1)}(\a)\right) & - C_1(f)\\
 & \le \sum_{j=2}^k  \Bigl( \hhat_f\left(f^{(n_j)}(\a)\right) + C_1(f) \Bigr) + \sum_{i=1}^r h(c_i) + \log(r-1),
\end{align*}
and then the transformation formula~$\hhat_f\circ f^{(n)}=d^n\hhat_f$
of Lemma~\ref{lem:canht}(b) gives 
\[
  d^{n_1} \hhat_f(\a) \le \Bigl( \sum_{j=2}^k d^{n_j} \Bigr)\hhat_f(\a) + kC_1(f)  + \sum_{i=1}^r h(c_i) + \log(r-1).
\]

Hence
\begin{equation}
\begin{split}
  \label{eqn:dn1sumdnihfa}
  \Bigl( d^{n_1} - \sum_{j=2}^k d^{n_j} \Bigr)\hhat_f(\a)
  &\le kC_1(f) + \sum_{i=1}^r h(c_i) + \log(r-1) \\
  &\le kC_1(f) + k h(F) + \log(k-1), 
  \end{split}
  \end{equation}
where we have used $r\le k$ and the definition of~$h(F)$.
We next use the fact that $n_1>n_2>\cdots>n_k\ge0$ to estimate
\begin{equation}
\begin{split}
  \label{eqn:dn1sumdnj}
  d^{n_1} - \smash[b]{ \sum_{j=2}^k d^{n_j} }
  &\ge d^{n_1} - d^{n_1-1} - d^{n_1-2} - \cdots - d^{n_1-k+1} \\
  &= \frac{d-2+d^{-k+1}}{d-1}\cdot d^{n_1}\\
  &\ge \begin{cases}
    2^{n_1-k+1} &\text{if $d=2$,} \\
    \frac12d^{n_1} &\text{if $d\ge3$.} 
    \end{cases}  
\end{split}
  \end{equation}
Using this bound in~\eqref{eqn:dn1sumdnihfa} gives a bound
for~$\hhat_f(\a)$ in terms of~$F$ and~$f$, and then using
$h\le\hhat_f+C_1(f)$ shows that~$h(\a)$ is bounded, which completes
of~(a).

\par\noindent(b)\enspace
Continuing with the computation from~(a), we note that~$n_1\ge k-1$,
so for $d\ge3$, the inequalities~\eqref{eqn:dn1sumdnihfa} and~\eqref{eqn:dn1sumdnj} yield
\[
  \hhat_f(\a) 
  \le \frac{2kC_1(f)}{d^{k-1}} + \frac{2k}{d^{k-1}} h(F) +  \frac{2\log(k-1)}{d^{k-1}}.
\]
Again using $h\le\hhat_f+C_1(f)$, together with the trivial estimates
\[
2k/d^{k-1}\le\frac{4}{3} \qquad \text{and} \qquad 
2\log(k-1)/d^{k-1}\le\frac{2}{9}\log2,
\]
 valid for $d\ge3$ and
$k\ge2$, we find that
\[
  h(\a) \le \frac{7}{3} C_1(f) + \frac{2k}{d^{k-1}} h(F) + \frac{2}{9}\log2.
\]

\par\noindent(c)\enspace
We are assuming that~$\a\in\K$ is wandering, so
Lemma~\ref{lem:canht}(d) says that $\hhat_f(\a)\ge C_2(\K,f) >0$.
Since we are further assuming that $0\notin\Orbit_f(\a)$, the
estimate~\eqref{eqn:dn1sumdnihfa} in~(a), combined with the lower
bound~\eqref{eqn:dn1sumdnj} yields
\begin{equation}
  \label{eqn:dn1d1d2dk1}
   d^{n_1} \le \frac{d-1}{d-2+d^{-k+1}} \cdot
  \frac{kC_1(f)+kh(F)+\log(k-1)}{C_2(\K,f)}.
\end{equation}
The right-hand side of~\eqref{eqn:dn1d1d2dk1} depends only
on~$\K$,~$f$, and~$F$, so it gives a bound for~$n_1$ depending only on
these quantities. Since $n_1>n_2>\cdots>n_k\ge0$, this completes the
proof that there are only finitely many $k$-tuples~$(n_1,\ldots,n_k)$
of distinct non-negative integers satisfying~\eqref{eqn:Ffn1fn2fnk},
and that the number of such $k$-tuples is bounded independently
of~$\alpha$. (This last statement would also follow from~(a), which
says that for all but finitely many $\a\in\K$, the
equation~\eqref{eqn:Ffn1fn2fnk} has no solutions.)
\qed

%%%%%%%%%%%%%%%%%%%%%%%%%%%%%%%%%%%%%%%%%%%%%%%%%%%%%%%%%%%%%%%%%%%%%%
\appendix

%%%%%%%%%%%%%%%%%%%%%%%%%%%%%%%%%%%%%%%%%%%%%%%%%%%%%%%%%%%%%%%%%%%%%%

%%%%%%%%%%%%%%%%%%%%%%%%%%%%%%%%%%%%%%%%%%%%%%%%%%%%%%%%%%%%%%%%%%%%%%
\section{Computation of Singular Points of $\bar C$}
\label{app:A}
%%%%%%%%%%%%%%%%%%%%%%%%%%%%%%%%%%%%%%%%%%%%%%%%%%%%%%%%%%%%%%%%%%%%%%
On the chart where $Z\ne0$, we dehomogenize to the equation $F(X)=cG(X)Y^m$.
A point is singular if and only if
\[
  F'(X)=cG'(X)Y^m\quad\text{and}\quad 0=cmG(X)Y^{m-1}.
\]
If $G(X)=0$, then $X=\b_j$ for some~$j$, and then from the equation of~$C$
we must have $F(\b_j)=0$, contradicting the assumption that~$F$ and~$G$ have no
common roots. So $Y=0$, and the equation of~$C$ forces $X=\a_k$ for some~$k$.
But we also have the condition $F'(X)=cG'(X)Y^m$, and evaluating at the point
$(\a_k,0)$ shows that $F'(\a_k)=0$. Thus~$\a_k$ needs to be a double root of~$F(X)$,
i.e., $e_k\ge2$.

We next consider the points on $\bar C$ satisfying
$Z=0$. Substituting~$Z=0$ into the equation of~$\bar C$ gives
$0=dbX^sY^m$, so either~$X=0$ or~$Y=0$.  So we are reduced to checking
the two points $[1,0,0]$ and $[0,1,0]$. To ease notation, we write
$\bar F(X,Z)$ and $\bar G(X,Z)$ for the homogenizations of~$F$ and~$G$.
We also let $M=m+d_G-d_F$, where we note that $M\ge2$ by the assumption
that $m\ge d_F+2$.

Around the point $[1,0,0]$, we dehomgenize $X=1$, so~$\bar C$ has the local
affine equation $\bar F(1,Z)Z^M-c\bar G(1,Z)Y^m=0$.
Using the facts tha $M\ge2$ and $m\ge2$, we see tha both the~$Y$ and~$Z$
derivatives of this equation vanish at~$[1,0,0]$, so~$[1,0,0]$ is always
a singular point. 

Finally, around the point $[0,1,0]$, we denomogenize $Y=1$, so~$\bar C$
has the local affine equation $\bar F(X,Z)Z^M-c\bar G(X,Z)=0$. Taking
the~$X$ and~$Z$ derivatives and subsituting~$[0,1,0]$, we see that~$[0,1,0]$
is a singular point if and only if $G_X(0,0)=G_Z(0,0)=0$. Hence~$[0,1,0]$
is singular if and only if~$d_G\ne1$.

\section*{Acknowledgement}

During the preparation of this work, A.~B. was supported by the NKFIH Grant K128088, 
A.~O. was supported by the
ARC Grant DP180100201,  I.~S.   was  supported   by the ARC Grant DP170100786
and J.~S. was supported by the  Simons Collaboration Grant \#241309.

\end{document}